\pdfoutput=1
\documentclass[numbers,webpdf,imanum]{ima-authoring-template}%

\graphicspath{{Fig/}}

\theoremstyle{thmstyletwo}%
\newtheorem{theorem}{Theorem}
\newtheorem{proposition}[theorem]{Proposition}%
\newtheorem{remark}{Remark}%
\newtheorem{lemma}{Lemma}%
\newtheorem{assumption}{Assumption}%
\newtheorem{corollary}{Corollary}%

\numberwithin{equation}{section}

\usepackage{stmaryrd}
\usepackage{mathtools}
\usepackage{bm}
\usepackage{microtype}

\providecommand{\jumptmp}[2]{#1\llbracket{#2}#1\rrbracket}
\providecommand{\jump}[1]{\jumptmp{}{#1}}

\def\tens#1{\pmb{\mathsf{#1}}}
\def\vec#1{\boldsymbol{#1}}

\def\R{\mathbb{R}}
\def\P{\mathbb{P}}
\def\Na{\mathbb{N}}
\def\Rd{\mathbb{R}^d}
\def\sym{\mathop{\mathrm{sym}}\nolimits}
\def\tr{\mathop{\mathrm{tr}}\nolimits}
\def\Rds{\mathbb{R}^{d \times d}_{\sym}}
\def\Rdst{\mathbb{R}^{d \times d}_{\sym,\tr}}
\def\fp{\,{:}\,} 

\def\rmdiv{\mathop{\mathrm{div}}\nolimits} 
\def\d{{\rm d}}
\def\Du{\BD(\bu)}

\def\bn{\vec{n}}
\def\bu{\vec{u}}
\def\bv{\vec{v}}
\def\bw{\vec{w}}

\def\BA{\tens{A}}
\def\BB{\tens{B}}
\def\BD{\tens{D}}

\def\BG{\tens{G}}
\def\BH{\tens{H}}

\def\BS{\tens{S}}
\def\BT{\tens{T}}
\def\BQ{\tens{Q}}

\def\LinftyT{L^\infty(0,T;L^2(\Omega)^d)}
\def\Linfty{L^\infty(I;L^2(\Omega)^d)}
\newcommand{\Lebs}[4][\Omega]{L^{#2}_{#3}(#1)^{#4}}
\newcommand{\Lp}[2][\Omega]{\Lebs[#1]{#2}{}{}}

\newcommand{\Lmean}[2][\Omega]{\Lebs[#1]{#2}{0}{}}

\usepackage[outline]{contour}

\def\Vh{\mathbb{V}^h}
\def\Vhdiv{\mathbb{V}^h_{\mathrm{div}}}
\def\Vht{\mathbb{V}^{h,\tau}}
\def\Mh{\mathbb{M}^h}
\def\Mht{\mathbb{M}^{h,\tau}}

\def\GRdt{\mu^{\mathrm{GR}}_{k_t+1}(\d t)}

\def\fD{\tens{\mathcal{D}}}
\def\fS{\tens{\mathcal{S}}}

\def\norm#1{\left|\!\left| #1 \right|\!\right|}
\def\dgrad{\mathcal{G}_h}
\def\ddiv{\mathcal{D}iv_h}
\def\Th{\mathcal{T}_{h}}
\def\Itau{\mathcal{I}_{\tau}}

\def\intF{\Gamma_h^{i}}
\def\Fh{\Gamma_h}

\def\avg#1{\{\!\!\{ #1 \}\!\!\}}

\def\leftjump{\left[\!\left[}
\def\rightjump{\right]\!\right]}

\definecolor{morado}{rgb}{0.65, 0.14, 0.78}

\newcommand{\WDG}{W^{1,p}(\mathcal{T}_h)}

\begin{document}

\DOI{DOI HERE}
\copyrightyear{2021}
\vol{00}
\pubyear{2021}
\access{Advance Access Publication Date: Day Month Year}
\appnotes{Paper}
\copyrightstatement{Published by Oxford University Press on behalf of the Institute of Mathematics and its Applications. All rights reserved.}
\firstpage{1}


\title[On the stability and convergence of DG schemes for incompressible flow]{On the stability and convergence of Discontinuous Galerkin schemes\\ for incompressible flow}

\author{Pablo Alexei Gazca--Orozco\ORCID{0000-0001-9859-4238} and Alex Kaltenbach*\ORCID{0000-0001-6478-7963}
\address{\orgdiv{Department of Mathematics}, \orgname{University of Freiburg}, \orgaddress{\street{Ernst--Zermelo--Stra\ss e}, \postcode{79104}, \state{Freiburg}, \country{Germany}}}}

\authormark{Pablo Alexei Gazca--Orozco and Alex Kaltenbach}

\corresp[*]{Corresponding author: \href{alex.kaltenbach@mathematik.uni-freiburg.de}{alex.kaltenbach@mathematik.uni-freiburg.de}}

\received{Date}{0}{Year}
\revised{Date}{0}{Year}
\accepted{Date}{0}{Year}


\abstract{The property that the velocity $\bu$ belongs to $\LinftyT$~is~an~essential~requirement~in~the~definition of energy solutions of models for incompressible fluids. It is, therefore, highly desirable that the solutions produced by discretisation methods are uniformly stable in the $\LinftyT$-norm. In this work, we establish that this is indeed the case for Discontinuous Galerkin (DG) discretisations (in time and space) of non-Newtonian models with $p$-structure, assuming that $p\geq \frac{3d+2}{d+2}$; the time discretisation is equivalent to the RadauIIA Implicit Runge--Kutta method. \textcolor{black}{We also prove (weak) convergence of the numerical scheme to the weak solution of the system; this type of convergence result for schemes based on quadrature seems to be new.} As an auxiliary result, we also derive  Gagliardo--Nirenberg-type inequalities on DG spaces, which might be of~independent~interest.}
\keywords{Discontinuous Galerkin; non-Newtonian implicitly constituted models; stability; convergence.}

\maketitle

\section{Introduction}

\subsection{Description of the model}

\qquad In this paper, we analyse the stability of non-conforming numerical schemes~for~a~system~describing the evolution of an incompressible non-Newtonian fluid. Namely, for a given spatial domain $\Omega \subseteq \R^d$, with $d\in\{2,3\}$, and a final time $0<T<\infty$, in the continuous setting, one looks for~a~velocity~vector~field $\bu\colon [0,T]\times \overline{\Omega} \to \R^d$, a pressure field $\pi\colon (0,T)\times \Omega \to \R$, and~a~(symmetric~and~traceless) stress tensor $\BS\colon (0,T)\times \Omega \to \Rdst$ such that
\begin{subequations}
\label{eq:continuous_PDE}
\begin{alignat}{2}
\begin{aligned}
\label{eq:continuous_balance}
\partial_t \bu - \rmdiv\BS + \rmdiv(\bu \otimes \bu) + \nabla \pi
 &= \bm{f} \qquad \quad &&\text{ in } (0,T)\times \Omega\,,\\
 \rmdiv\bu &= 0 \qquad \quad &&\text{ in }{(0,T)\times \Omega}\,,\\
  \bu &= \bm{0} \qquad \quad &&\text{ on }(0,T)\times \partial\Omega\,,\\
 \bu(0,\cdot) &= \bu_0 \qquad \quad &&\text{ in }\Omega\,,
\end{aligned}
\end{alignat}
where the initial velocity vector field $\bu_0\colon \Omega \to \Rd$ and the body force $\bm{f}\colon (0,T)\times \Omega \to \Rd$~are~given. To close the system, we consider an implicit constitutive law of the form
\begin{align}
	\label{eq:continuous_constitutive_law}
	\BG(\BS, \Du) = \bm{0} \qquad
	\text{ in } (0,T)\times \Omega\,,
\end{align}
\end{subequations}
where $\Du = \frac{1}{2}(\nabla\bu + \nabla\bu^\top)\colon(0,T)\times \Omega\to\Rds $ denotes the strain rate tensor, i.e., symmetric part of the velocity gradient, and $\BG\colon \Rds \times \Rds \to \Rds$ is a locally Lipschitz function such that $\BG(\bm{0},\bm{0})=\bm{0}$ and such that it defines a $p$-coercive~graph~for~$p>1$, in the sense that there exist two constants $c_1, c_2>0$ such that for every $(\BA,\BB)\in \Rds\times \Rds$, it holds that
\begin{align}\label{eq:coercivity}
\BG(\BA,\BB) = \bm{0} \qquad
\Longrightarrow \qquad
\BA\fp\BB \geq c_1(|\BA|^{p'} + |\BB|^p) - c_2\,.
\end{align}
\textcolor{black}{Here $p'\coloneqq \tfrac{p}{p-1}$ denotes the Hölder conjugate exponent of $p$.}
Such a class of constitutive relations captures many models that are popular in applications. Prototypical examples that, in addition, define a \emph{monotone graph} include fluids with power-law structure
\begin{subequations}\label{eq:power_law}
	\begin{gather}
\BG(\BS,\BD) \coloneqq  \BS - K_\star(1 + \Gamma_\star |\BD|^2)^{\smash{\frac{p-2}{2}}}\BD
\qquad K_\star,\Gamma_\star>0\,,\; p>1\,, \label{eq:power_lawA}\\
\BG(\BS,\BD) \coloneqq  K_\star(1 + \Gamma_\star |\BS|^2)^{\smash{\frac{p'-2}{2}}}\BS - \BD
\qquad K_\star,\Gamma_\star>0\,,\; p>1\,,
\end{gather}
\end{subequations}
or viscoplastic Bingham fluids
\begin{align}\label{eq:bingham_implicit}
	\BG(\BS,\BD) \coloneqq  (|\BS| -  \tau_\star)^+ \BS
- 2\nu_\star(\tau_\star + (|\BS|-\tau_\star)^+)\BD
\qquad \nu_\star > 0\,,\; \tau_\star \geq 0\,,
\end{align}
where $(\cdot)^+ \!\coloneqq \! (s\mapsto\max\{s,0\})\colon \!\mathbb{R}\!\to\! \mathbb{R}$. This relation is more commonly written~in~terms~of~the~dichotomy
\begin{align}\label{eq:bingham_dichotomy}
  \renewcommand{\arraystretch}{1.2}
\left\{
	\begin{array}{ccc}
		|\BS|\leq \tau_\star & \Longleftrightarrow & \BD = \bm{0}\,, \\[1mm]
|\BS|> \tau_\star & \Longleftrightarrow & \BS = 2\nu_\star \BD + \displaystyle\tfrac{\tau_\star}{|\BD|}\BD\,.
	\end{array}
\right.
\end{align}
Note that while it is not possible to write the relation \eqref{eq:bingham_dichotomy} in terms of a single valued function~$\BS(\BD)$, within the implicit framework, one can express it in terms of elementary functions without issue. We note further that the Newtonian constitutive relation is of course also considered here (e.g., take $\tau_\star = 0$~in \eqref{eq:bingham_implicit} or $p=2$ in \eqref{eq:power_law}). We refer to \cite{BMR.2020,BMM.2021}, for an in-depth discussion of the different models that can be described with such monotone constitutive relations and the corresponding PDE analysis. 

The framework of implicit constitutive relations also includes non-monotone relations that can describe hysteretic behaviour, e.g.,
\begin{align}\label{eq:non_monotone}
	\BG(\BS,\BD) = \big[a(1 + b|\BS|^2)^{\smash{\frac{q-2}{2}}} + c\big]\BS
	- \BD 
	\qquad a,b,c>0\,,\; q\in \R\, .
\end{align}
which for $q<0$, in general, is \emph{non-monotone} (see \cite{LRR.2013} for details), but has, nevertheless,~been~shown to be thermodynamically consistent \cite{JP.2018}; see also \cite{JMPT.2019}, for insightful numerical experiments.

In this work, we concentrate on non-conforming discretisations of the problem \eqref{eq:continuous_PDE}; namely, a discontinuous Galerkin in time method $\mathrm{dG}(k)$ and a Discontinuous Galerkin discretisation in space that can, in particular, be taken to be a Local Discontinuous Galerkin (LDG) method or an Interior Penalty (IP) method (possibly incomplete). The dG time discretisation we consider here can be shown to be equivalent to a RadauIIA Implicit Runge--Kutta scheme \cite{MN.2006}, which, due to its L-stability, is popular in applications modelled by parabolic problems; the Discontinuous Galerkin methodology allows the straightforward development of arbitrary high-order time discretisations and has been intensely studied in recent decades; see, e.g., \cite{SS.2000,AM.2004,SW.2010,Tho.2007,SW.2019}.
Regarding the spatial discretisation, in the case of incompressible fluid models such as \eqref{eq:continuous_PDE}, one has the additional concern of the preservation of the divergence-free constraint  \eqref{eq:continuous_balance}$_2$ at the discrete level. In recent years, the importance of this has been recognised and schemes that lead to point-wise divergence-free approximations~have~many~desirable~qualities, such as pressure robust error estimates; see \cite{JLMNR.2017} for more details. One of the main ways of obtaining exactly divergence-free approximations is to relax the conformity requirement and employ a finite element space for the velocity that is $H(\rmdiv;\Omega)$-conforming only. This non-conformity is, then, handled by including DG terms in the formulation; see, e.g., \cite{CKS.2007,SLLL.2018} for the Newtonian case. While this is one of our main motivations, here we will analyse more general discretisations that might not~enforce~the~divergence~constraint~exactly.

Given the highly non-linear nature of the models considered here, deriving~error~estimates~seems out of reach. In such cases, one can turn instead to proving~(weak)~convergence~(possibly~of~a~subsequence) to minimal regularity solutions by using compactness arguments. A crucial step in such arguments is to establish stability of the corresponding discrete scheme, from which one then extracts (weakly) converging subsequences. This approach was taken in \cite{ST.2019,FGS.2020} for conforming-in-space discretisations of implicitly constituted models; for the case with explicit constitutive relations (and implicit Euler in time), see \cite{BKR.2021,KR.2022}. In the setting considered here, the coercivity condition \eqref{eq:coercivity} results in a stability estimate that guarantees the uniform boundedness of the velocity approximations in $L^p(0,T;W^{1,p}(\Omega)^d)$ (or, more precisely, on its broken counterpart) and of the stress approximations in $L^{p'}((0,T)\times \Omega)^{d\times d}$. This is, however, not enough as the usual notions of energy solutions for incompressible models require also that $\bu \in L^\infty(0,T;L^2(\Omega)^d)$; among other things, this condition is useful because (see, e.g., \cite{ST.2019} for more details):
\begin{itemize}
		\item Together with a Gagliardo--Nirenberg-type interpolation inequality, cf. \cite[Theorem I.2.1]{dibene}, it implies that $$\bu \in L^{\frac{p(d+2)}{d}}((0,T)\times\Omega)^d\,,$$ which, in turn, implies, e.g., that if $p\geq \frac{3d+2}{d+2}$ (and so, in particular, for the Newtonian problem in 2D), then the velocity is an admissible test function in the balance of momentum and, which guarantees an energy identity and, thus, uniqueness of solutions;\vspace{1mm}\enlargethispage{4mm}
	\item It is used when proving that $$\bu \in C_w^0([0,T];L^2(\Omega)^d)\,,$$ meaning that the initial condition is a priori meaningful in this weak sense, but in fact this allows one to prove that $$\lim_{t\to 0}\|\bu(t) - \bu_0\|_{L^2(\Omega)}= 0\,.$$
\end{itemize}

It is, therefore, highly desirable that the discretisation methods produce solutions which are also uniformly stable in $\LinftyT$. By testing the dG-in-time discretised system with the solution, it is straightforward (cf.\ Lemma \ref{lem:apriori}) to prove $L^2(\Omega)^d$-stability at the partition points $\{t_j\}_{j\in \{0,\dots,N_\tau\}}$.~However, this only yields the desired $\LinftyT$-bound in the lowest order case $\mathrm{dG}(0)$ (i.e.,~implicit~Euler), since the function is piece-wise constant in time. In general, when working with general dG in time discretisations, one can only guarantee stability in $L^{2p}(0,T;L^2(\Omega)^d)$; see \cite{W.2010} and \cite{AGF.2023}, for the spatially conforming and non-conforming cases, respectively. Thus, in general, one would obtain convergence to a weaker notion of solution that might not be unique even when $p=2=d$. Chrysafinos and Walkington (cf.\ \cite{CW.2010b}) proved, however, with the help of Ladyzhenskaya's inequality, that for spatially conforming discretisations, one can still obtain $\LinftyT$-stability for the Newtonian problem ($p=2$) in two spatial dimensions ($d=2$). One of the contributions of this work is the extension of this result to the non-Newtonian and non-conforming setting; in particular, we establish that if $p\geq \frac{3d+2}{d+2}$ (i.e., when the velocity is an admissible test function), DG discretisations are stable also in $\LinftyT$. 

{\color{black}
    In essence,  the argument is based on testing the equation with the discrete solution multiplied by an exponential function; when looking at, say, the time interval $(t_{j-1},t_j)$, $j\in\{1,\dots,N_\tau\}$, then testing, for every $j\in\{1,\dots,N_\tau\}$, with $(t\mapsto e^{-\lambda(t-t_{j-1})}\bu(t))\colon (0,T)\to W^{1,p}_0(\Omega)^d$,   $\lambda>0$,  for every $j\in\{1,\dots,N_\tau\}$, results in the following for the time-derivative term:
    \begin{align*}
        \int_{t_{j-1}}^{t_{j}} \int_\Omega \partial_t\bu \cdot \bu e^{-\lambda (t-t_{j-1})}\,\mathrm{d}x\,\mathrm{d}t&
        =
        \frac{\lambda}{2}\int_{t_{j-1}}^{t_j} e^{-\lambda(t-t_{j-1})} \|\bu(t)\|^2_{L^2(\Omega)}\,\mathrm{d}t
        \\&\quad + \frac{e^{-\lambda(t_j - t_{j-1})}}{2} \|\bu(t_{j})\|^2_{L^2(\Omega)}
        - \frac{1}{2} \|\bu(t_{j-1})\|^2_{L^2(\Omega)}\,.
    \end{align*}
    Since the $L^2(\Omega)$-norm of the velocity is uniformly bounded at the partition points $\{t_j\}_{j\in \{0,\dots,N_\tau\}}$ (and since the remaining terms in the discrete formulation can be controlled), setting $\lambda>0$ as the inverse time step and applying an inverse estimate in time to the first term in the right-hand-side, yields the desired $L^\infty(0,T;L^2(\Omega)^d)$-bound. The problem is that $(t\mapsto e^{-\lambda(t-t_{j-1})}\bu(t))\colon (0,T)\to W^{1,p}_0(\Omega)^d$,   $\lambda>0$, $j\in\{1,\dots,N_\tau\}$, are not admissible test functions in the discrete formulation, because it is not a polynomial in time. To circumvent this issue, an discrete approximation of the desired test function --the so-called exponential time interpolant-- was introduced in \cite{CW.2010b}; see also the works \cite{Chry.2010,CK.2012,Chry.2013,KHR.2023}, where this tool has been applied in other contexts. In this work, we show that this exponential time interpolant has also the required properties to carry out the argument in our setting, namely non-Hilbertian ($p\neq 2$) and a time discretisation method based on quadrature.

    Another contribution of this work is a proof of (weak) convergence of the discrete solutions to the weak solution of the system based on a compactness argument. We should note that the time discretisation method we consider is based on writing the time integrals in terms of a discrete measure $\GRdt$ based on quadrature around Gauss--Radau nodes (this is what results in the equivalence~to~the~RadauIIA~method). This leads to the difficulty that (as opposed usual approaches) the uniform bounds at hand contain norms in which the time integral is not exact --thus, preventing the direct use of the Banach--Alaoglu theorem. To the best of our knowledge, this is the first such result for time discretisations of non-linear problems.
}

Another important step in the proof is the application of a Gagliardo--Nirenberg inequality on DG spaces, which is needed since the numerical solutions are discontinuous across elements, which we also derive and is to the best of our knowledge also new; \textcolor{black}{see also \cite{KHR.2023}, where the inequality was recently derived for $p=2$ and more restrictive assumptions on the mesh}. 

\textit{This article is organized as follows:} In Section \ref{sec:preliminaries}, we introduce the employed notation, the basic assumptions on the mesh regularity, and the relevant spaces and operators from~DG~theory.~In~Section~\ref{sec:gagliardo}, we establish a discrete Gagliardo--Nirenberg-type inequality on DG spaces. In Section, \ref{sec:parabolic_interpolation}, using the discrete Gagliardo--Nirenberg-type inequality from Section \ref{sec:gagliardo}, we derive several parabolic discrete interpolation inequalities. These discrete parabolic interpolation inequalities are employed in Section~\ref{sec:stablity} to prove the $L^\infty(0,T,L^2(\Omega)^d)$-stability of discontinuous Galerkin schemes for incompressible flows. In Section \ref{sec:convergence}, we establish (weak) convergence of the discrete solutions to the weak solution of the system.\enlargethispage{6mm}

\section{Preliminaries}\label{sec:preliminaries}

		\hspace{5mm}Throughout the entire article, if not otherwise specified, we always denote by ${\Omega\subseteq \mathbb{R}^d}$,~${d\in\mathbb{N}}$, a bounded polyhedral Lipschitz domain with outward-pointing unit vector field $\bn\colon \partial\Omega\to \mathbb{S}^{d-1}$.
        Then, the time interval will be denoted by $I\coloneqq (0,T)$, $0<T<\infty$, and  the parabolic~cylinder~by~$Q\coloneqq  I\times \Omega$.\linebreak For $p\in [1,\infty]$ and $k\in\Na$, we will employ standard notation for Lebesgue $L^p(\Omega)$, Sobolev $W^{k,p}(\Omega)$, and Bochner--Sobolev $L^p(I;W^{k,p}(\Omega))$ spaces throughout. For $p\in [1,\infty)$ and $k\in \Na$, we~denote~by~$W_0^{k,p}(\Omega)$, the closure of the space of smooth functions on $\Omega$ with compact support, with~respect~to~the~$W^{k,p}(\Omega)$- norm. The subspace of $L^p(\Omega)$ functions with zero mean~will~be~denoted~by~$L_0^p(\Omega)$.

        \subsection{Mesh regularity}

				        \hspace{5mm}In this subsection, we propose a set of assumptions on the family of partitions $\{\Th\}_{h\in (0,1]}$,  which are required  to apply the theory developed in this paper. These assumptions correspond to the choice~in~\cite{BO.2009}.

        Let $\{\Th\}_{h\in (0,1]}$ be a family of partitions of the closure $\overline{\Omega}$ into convex polyhedral elements, which are affine images of a set of reference polyhedra. More precisely, we assume that there exists a finite number of convex reference polyedra $\widehat{K}_1,\dots,\widehat{K}_N$, such that $\vert\widehat{K}_N\vert = 1 $ for $i = 1,\dots, N$, and that for each $K \in \Th$, there exists a reference element $\widehat{K}_i$ for some $i\in \{1,\dots, N\}$ and an invertible affine map $F_K\colon \widehat{K}_i\to K$ such that $K = F_K (\widehat{K}_i)$. The symbol $h>0$ denotes the maximal mesh size, i.e., if we define $h_K\coloneqq \text{diam}(K)$ for every $K\in \Th$, then we have that $h=\max_{T\in \Th}{h_K}$.  Without loss of generality, we assume that $h\in  (0, 1]$. 
        
        \hspace*{-5mm}We will provide further assumptions on the mesh regularity in the course of this section.

				        We define the sets of $(d-1)$-dimensional faces $\Gamma_h$, interior faces $\Gamma_h^{i}$, and boundary faces $\Gamma_h^{\partial}$ of the partition $\Th$ by
        \begin{align*}
            \Gamma_h&\coloneqq \Gamma_h^{i}\cup \Gamma_h^{\partial}\,,\\[-0.5mm]
             \Gamma_h^{i}&\coloneqq  \{K\cap K'\mid K,K'\in \Th\,,\text{dim}_{\mathscr{H}}(K\cap K')=d-1\}\,,\\[-0.5mm]
             \Gamma_h^{\partial}&\coloneqq\{K\cap \partial\Omega\mid K\in \Th\,,\text{dim}_{\mathscr{H}}(K\cap \partial\Omega)=d-1\}\,,
        \end{align*}
        where for every  $S\subseteq \mathbb{R}^d$, we denote by $\text{dim}_{\mathscr{H}}(S)\coloneqq\inf\{d'\geq 0\mid \mathscr{H}^{d'}(S)=0\}$, the~Hausdorff~dimension. The (local) mesh-size function $h_{\mathcal{T}}\colon \Omega\to \mathbb{R}$ for every element $K\in \Th$ is defined by $h_{\mathcal{T}}|_K\coloneqq h_K $.
        The (local) face-size function $h_{\Gamma}\colon \Gamma_h\to \mathbb{R}$ for every facet $F\in \Gamma_h$ is defined by $h_{\Gamma}|_F\coloneqq h_F \coloneqq \text{diam}(F)$.\enlargethispage{1mm}

            \begin{assumption}[Mesh quality; cf. \cite{BO.2009}]\label{assum:mesh}
							We assume that $\{\Th\}_{h\in (0,1]}$ satisfies the following conditions:
                \begin{itemize}[{(iii)}]
                    \item[(i)] \textup{Shape Regularity.} There exist constants $c_1,c_2>0$ such that for every $K\in \mathcal{T}_h$ and~$h\in (0,1]$,~it~holds that 
                    \begin{align*}c_1\, h_K^d\leq \vert K\vert\leq  c_2\, h_K^d\,;
                    \end{align*}

                    \item[(ii)] \textup{Contact Regularity.} There exists a constant $c_3>0$ such that for every $F\in \Gamma_h$ with $F\subseteq \overline{K}$ for some $K\in \mathcal{T}_h$ and $h\in (0,1]$, it holds that
                    \begin{align*}c_3\, h_K^{d-1}\leq \mathscr{H}^{d-1}(F)\,;
                    \end{align*}

                    \item[(iii)] \textup{Submesh condition.} There exists a shape-regular, conforming, matching simplicial submesh $\widetilde{\Th}$  (without hanging nodes, edges, etc.) such that the following assumptions are satisfied:
                \begin{itemize}[{3.}]
									                    \item[1.] For each $\widetilde{K}\in \widetilde{\Th}$, there exists $K\in \Th$ such that $\widetilde{K}\subseteq K$;\vspace*{1mm}
                    \item[2.] The family $\{\widetilde{\Th}\}_{h\in (0,1]}$ satisfies (i) and (ii);\vspace*{1mm}\vspace{-0.5mm}
                    \item[3.] There exists a constant $\tilde{c}>0$ such that for any $\widetilde{K}\in \smash{\widetilde{\Th}}$, $K\in \Th$ with $\widetilde{K}\subseteq K$,~it~holds~that $${h_K \leq \tilde{c}\, h_{\widetilde{K}}}\,.$$
                \end{itemize}
                \end{itemize}
            \end{assumption}


\subsubsection{Broken function spaces and projectors}

\hspace{5mm}For every $k \in \Na_0$~and~${K\in \mathcal{T}_h}$,
we denote by $\mathbb{P}_k(K)$, the space of
polynomials of degree at most $k$ on $K$.  Then, for given $k\in \Na_0$, we define the space of \textit{broken polynomials of global degree at most $k$}
\begin{align*}
        \mathbb{P}_k(\mathcal T_h)&\coloneqq\big\{v_h\in L^\infty(\Omega)\mid v_h|_K\in \mathbb{P}_k(K)\text{ for all }K\in \mathcal{T}_h\big\}\,.
\end{align*}
In addition, for
given~$p\in (1,\infty)$,
we define the \textit{broken Sobolev space}
\begin{align*}
        W^{1,p}(\mathcal T_h)&\coloneqq\big\{w_h\in L^p(\Omega)\mid w_h|_K\in W^{1,p}(K)\text{ for all }K\in \mathcal{T}_h\big\}\,.
\end{align*}

For  each $w_h\!\in\! \WDG$, we denote by $\nabla_h w_h\!\in\! L^p(\Omega)^d$, 
the \textit{local gradient}, for~every~${K\!\in\!\mathcal{T}_h}$,~defined by
$(\nabla_h w_h)|_K\!\coloneqq\!\nabla(w_h|_K)$~for~all~${K\!\in\!\mathcal{T}_h}$.
For each $K\!\in\! \mathcal{T}_h$,~${w_h\!\in\! \WDG}$~admits~a~trace~${\textrm{tr}^K(w_h)\!\in\! L^p(\partial K)}$. For each face
$F\in \Gamma_h$ of a given element $K\in \mathcal{T}_h$, we define~this~interior trace by
$\smash{\textup{tr}^K_F(w_h)\in L^p(F)}$. Then, given some multiplication operator~${\odot\colon \mathbb{R}^m\times \mathbb{R}^d\to \mathbb{R}^l}$,~${m,l\in \Na}$, for
 every $w_h\in \WDG$ and interior faces $F\in \Gamma_h^{i}$ shared by
adjacent elements $K^-_F, K^+_F\in \mathcal{T}_h$,
we denote by
\begin{align*}
  \{w_h\}_F&\coloneqq\tfrac{1}{2}\big(\textup{tr}_F^{K^+}(w_h)+
  \textup{tr}_F^{K^-}(w_h)\big)\in
  L^p(F)\,,\\
  \llbracket w_h\odot \bn\rrbracket_F
  &\coloneqq\textup{tr}_F^{K^+}(w_h)\odot\bn^+_F+
    \textup{tr}_F^{K^-}(w_h)\odot\bn_F^- 
    \in L^p(F)\,,
\end{align*}
the \textit{average} and \textit{jump}, respectively, of $w_h$ on $F$.
    Moreover,  for every $w_h\in \WDG$ and boundary faces $F\in \Gamma_h^{\partial}$, we define \textit{boundary averages} and 
    \textit{boundary jumps}, respectively, by
    \begin{align*}
      \{w_h\}_F&\coloneqq\textup{tr}^\Omega_F(w_h) \in L^p(F)\,, \\
      \llbracket w_h\odot\bn\rrbracket_F&\coloneqq
      \textup{tr}^\Omega_F(w_h)\odot\bn \in L^p(F)\,.
    \end{align*}
    If there is no
    danger~of~confusion, we will omit the index $F\in \Gamma_h$. In particular,  if we interpret jumps and averages as global functions defined on the whole of $\Gamma_h$.
    Apart from that, for every $w_h\in \WDG$, we  introduce the \textit{DG norm} via
    \begin{align*}
    \|w_h\|_{h,p}\coloneqq\Big(\|\nabla_hw_h\|_{L^p(\Omega)}^p+\big\|h^{-1/p'}_\Gamma\jump{w_h\bn}\big\|_{L^p(\Gamma_h)}^p\Big)^{1/p}\,,
    \end{align*}
which turns $\WDG$ into a Banach space. 
With this norm, cf.~\cite[Lm. A.9]{DKRI14},~for~every~${w_h\in \WDG}$,  there holds the \textit{discrete Poincar\'e inequality}
\begin{align}\label{eq:poincare}
\|w_h\|_{L^p(\Omega)} \lesssim \|w_h\|_{h,p}\,.
\end{align}
		Whenever we write $A \lesssim B$, it is meant that $A \leq c\, B$ with a constant $c>0$ that might depend on the domain, polynomial degree and/or shape regularity, but is independent of the discretisation parameters (i.e., the mesh size $h>0$ or the time step size $\tau>0$).\enlargethispage{2mm}

\section{Discrete Gagliardo--Nirenberg-type inequality}\label{sec:gagliardo}
    \qquad In this section, we derive a discrete Gagliardo--Nirenberg-type inequality;
    {\color{black}
    see also \cite{KHR.2023}, where the same inequality was derived for $p=2$ and more restrictive assumptions on the mesh.
    }
     Key ingredient is 
    the  quasi-interpolation operator  $Q_h\colon  \P_k(\Th)\to \P_1(\smash{\widetilde{\Th}})\cap W^{1,\infty}(\Omega)$, where $\smash{\widetilde{\Th}}$ denotes the simplicial submesh in Assumption \ref{assum:mesh} (c), introduced~in~\cite{BO.2009}, and its approximation and stability properties~on~DG~spaces:

    \begin{lemma}\label{lem:scott_zhang_stable}
			       Let $p\in [1,\infty)$ and $k\in \mathbb{N}_0$. Then, 
        for every $v_h\in \P_k(\Th)$, it holds that 
\begin{align}\label{eq:scott_zhang_stable}
            \|\nabla Q_hv_h\|_{L^p(\Omega)}\lesssim
            \|v_h\|_{h,p}\,.
        \end{align}
    \end{lemma}\vspace{-3mm}

    \begin{proof}
        See \cite[Thm. 3.1, (3.11)]{BO.2009}.
    \end{proof}

   \begin{lemma}\label{lem:scott_zhang_approx}
		 Let $p,s\in [1,\infty)$ and  $k\in\mathbb{N}_0$.
        Then,
        for every $v_h\in \P_k(\Th)$ and $K\in \Th$, it holds that
        \begin{align} \label{eq:scott_zhang_approx}
						            \|v_h-Q_hv_h\|_{L^s(K)}\lesssim 
            h_K^{1+d(\frac{1}{s}-\frac{1}{p})}\|v_h\|_{h,p,\omega_K}\,,
        \end{align}
        where $\omega_K\coloneqq \bigcup\{K'\in\Th\mid K'\cap K\neq\emptyset\} $ and, for every $w_h \in W^{1,p}(\Th)$, we define
        \begin{align}
        \smash{\|w_h\|_{h,p,\omega_K}\coloneqq (\|\nabla_hw_h\|_{L^p(\omega_K)}^p+\|h^{-1/p'}_\Gamma\jump{w_h\bn}\|_{L^p(\Gamma_h\cap \omega_K)}^p)^{1/p}}.
\end{align}
        In particular, for every $v_h\in \P_k(\Th)$,
        it holds that
        $$
						            \|v_h-Q_hv_h\|_{L^p(\Omega)}\lesssim \|h_{\mathcal{T}}v_h\|_{h,p}\,.
        $$
    \end{lemma}

    \begin{proof}
        See \cite[Thm. 3.1, (3.7) \& (3.10)]{BO.2009}. 
    \end{proof}

    \begin{corollary}\label{cor:scott_zhang_stable}
			       Let $p\in [1,\infty)$ and $k\in\mathbb{N}_0$. Then, 
        for every $v_h\in \P_k(\Th)$ and $K\in \Th$, it holds that
        $$
					            \| Q_hv_h\|_{L^p(\Omega)}+\| v_h-Q_hv_h\|_{L^p(\Omega)}\lesssim  
            \|v_h\|_{L^p(\Omega)}\,.
        $$
         In particular, for every $v_h\in \P_k(\Th)$,
        it holds that
         $$
					            \| Q_hv_h\|_{L^p(K)}+\| v_h-Q_hv_h\|_{L^p(K)}\lesssim  
            \|v_h\|_{L^p(\omega_K)}\,.
        $$
    \end{corollary}

    \begin{proof}\let\qed\relax
        Using the $L^p$-approximation property of $Q_h$ for $s=p$ (cf. Lemma \ref{lem:scott_zhang_approx}), the inverse inequality (cf. \cite[Ex. 12.3]{EG21}), and the discrete trace inequality (cf.~\cite[Lm. 12.8]{EG21}), we find that
        \begin{align*}
						            \| Q_hv_h\|_{L^p(K)}+\| v_h-Q_hv_h\|_{L^p(K)}&\lesssim \|  v_h\|_{L^p(K)} + \|  v_h -Q_hv_h\|_{L^p(K)}\\&
            \lesssim \|  v_h\|_{L^p(K)}+h_K\,\|  v_h\|_{h,p,\omega_K}
						\lesssim \|  v_h\|_{L^p(\omega_K)}\,.\tag*{\qedsymbol}
        \end{align*}
    \end{proof}\vspace*{-10mm}\enlargethispage{8mm}
    
    \begin{lemma}[Gagliardo--Nirenberg]\label{lem:gagliardo}
			        Let $p,q\in [1,\infty)$ and $k\in\mathbb{N}_0$. Then, 
        for every $v_h\in \P_k(\Th)$, it holds that 
        \begin{align*}
						            \|v_h\|_{L^s(\Omega)}\lesssim 
            \|v_h\|_{h,p}^\gamma\|v_h\|_{L^q(\Omega)}^{1-\gamma}\,,
        \end{align*}
         where $s\in [1,\infty)$ and $\gamma\in [0,1]$ satisfy
        \begin{align}
            \gamma=\frac{\frac{1}{q}-\frac{1}{s}}{\frac{1}{q}+\frac{1}{d}-\frac{1}{p}}\,.\label{eq:gamma}
        \end{align}
    \end{lemma}

    Analogously to \cite[Thm. I.2.1]{dibene}, for each $d\ge 2$, the admissible range for $p,q,s\in [1,\infty)$ and $\gamma\in [0,1]$ satisfying \eqref{eq:gamma}, setting $p_*\coloneqq \smash{\frac{dp}{d-p}}$ if $p<d$, is given by:\vspace{-3mm}
		    \begin{subequations}\label{eq:admissibility}
    \begin{alignat}{4}
        &\text{if }p\in [1,d):\quad &&\gamma\in [0,1]\qquad &&\text{and}\qquad &&s\in \begin{cases}
            [q,p_*]&\text{if }q\in [1,p_*]\\
            [p_*,q]&\text{if }q\in [p_*,\infty)
        \end{cases}\,,\label{eq:admissibility.1}\\
				&\text{if }p\in [d,\infty):\quad &&s\in [q,\infty)\qquad &&\text{and}\qquad &&\gamma\in  \big[0,\tfrac{dp}{dp+q(p-d)}\big)\,.\label{eq:admissibility.2}
    \end{alignat}
		\end{subequations}

    \begin{proof}[Proof (of Lemma \ref{lem:gagliardo}).]
    
        To begin with, we observe that
        \begin{align}\label{eq:gagliardo.1}
						 \|v_h\|_{L^s(\Omega)}\leq \|Q_hv_h\|_{L^s(\Omega)}+\|v_h-Q_hv_h\|_{L^s(\Omega)}\eqqcolon I_h^1+I_h^2\,.
        \end{align}
        
        As a result, it suffices to estimate $I_h^1$ and $I_h^2$ separately:

				\textit{ad $I_h^1$.} Using the classical Galgiardo--Nirenberg inequality \cite{Nir59}, the discrete Poincar\'e~inequality~\eqref{eq:poincare}, the DG-stability of $Q_h$ (cf.~Lemma~\ref{lem:scott_zhang_stable}),   and the $L^q$-stability~property~of~$Q_h$~(cf.~Corollary~\ref{cor:scott_zhang_stable}),~we~deduce~that
        \begin{align}\label{eq:gagliardo.2}
            \begin{aligned}
							I_h^1&\lesssim \,(\| Q_hv_h\|_{L^p(\Omega)}+\|\nabla Q_hv_h\|_{L^p(\Omega)})^{\gamma}\|Q_hv_h\|_{L^q(\Omega)}^{1-\gamma}
                \\
								                &\lesssim \|Q_h v_h\|_{h,p}^{\gamma}\|Q_hv_h\|_{L^q(\Omega)}^{1-\gamma}
                \\&\lesssim \|v_h\|_{h,p}^{\gamma}\|v_h\|_{L^q(\Omega)}^{1-\gamma}\,.
            \end{aligned}
        \end{align}

        \textit{ad $I_h^2$.} Using Lemma \ref{lem:scott_zhang_approx}, \cite[Ex. 12.4]{EG21} for all $K\in \Th$ and $\widetilde{K}\in \widetilde{\Th}$, that $h_K\leq \tilde{c}\,h_{\widetilde{K}}\leq \tilde{c}\,h_K$~for~all~$K\in \Th$ and $\widetilde{K}\in \widetilde{\Th}$ with $\widetilde{K}\subseteq K$ (cf. Assumption \ref{assum:mesh} (c) 3.), that $\textup{card}(\{\widetilde{K}\in \smash{\widetilde{\Th}}\mid \widetilde{K}\subseteq K\})\lesssim 1$ for all $K\in \Th$ (cf. \cite[Lm. 1.40]{EP12}), Corollary \ref{cor:scott_zhang_stable}, and that $\sum_{i\in \mathbb{L}}{\vert a_i\vert^s}\leq (\sum_{i\in \mathbb{L}}{\vert a_i\vert})^s$ 
        for any finite subset $\mathbb{L}\subseteq \mathbb{N}$ and finite sequence $(a_i)_{i\in \mathbb{L}}\subseteq \mathbb{R}$,
         we find that
        \begin{align}\label{eq:gagliardo.3}
             \begin{aligned}
						            (I_h^2)^s&\leq \sum_{K\in \Th}{\Big(\|v_h-Q_hv_h\|_{L^s(K)}^{\gamma}\|v_h-Q_hv_h\|_{L^s(K)}^{1-\gamma}}\Big)^s
                  \\&\lesssim \sum_{K\in \Th}{\Bigg(\Big(h_K^{1+d(\frac{1}{s}-\frac{1}{p})}\|v_h\|_{h,p,\omega_K}\Big)^{\gamma}\bigg(\sum_{\widetilde{K}\in \widetilde{\Th};\widetilde{K}\subseteq K}{\|v_h-Q_hv_h\|_{L^s(\widetilde{K})}^s}\bigg)^{\smash{\frac{1-\gamma}{s}}}\Bigg)^s}
                   \\&\lesssim \sum_{K\in \Th}{\Bigg(\Big(h_K^{1+d(\frac{1}{s}-\frac{1}{p})}\|v_h\|_{h,p,\omega_K}\Big)^{\gamma}\bigg(\sum_{\widetilde{K}\in \widetilde{\Th};\widetilde{K}\subseteq K}{h_{\widetilde{K}}^{d(\frac{1}{s}-\frac{1}{q})s}\|v_h-Q_hv_h\|_{L^q(\widetilde{K})}^s\bigg)^{\smash{\frac{1-\gamma}{s}}}}\Bigg)^s}
            \\&\lesssim \sum_{K\in \Th}{\bigg(\Big(h_K^{1+d(\frac{1}{s}-\frac{1}{p})}\|v_h\|_{h,p,\omega_K}\Big)^{\gamma}\Big(h_K^{d(\frac{1}{s}-\frac{1}{q})}\|v_h-Q_hv_h\|_{L^q(K)}\Big)^{1-\gamma}\bigg)^s}
            \\&\lesssim \sum_{K\in \Th}{\bigg(h_K^{(1+d(\frac{1}{s}-\frac{1}{p}))\gamma+d(\frac{1}{s}-\frac{1}{q})(1-\gamma)}}\|v_h\|_{h,p,\omega_K}^{\gamma}\|v_h\|_{L^q(\omega_K)}^{1-\gamma}\bigg)^s
            \\&\lesssim \bigg(\sum_{K\in \Th}{h_K^{(1+d(\frac{1}{s}-\frac{1}{p}))\gamma+d(\frac{1}{s}-\frac{1}{q})(1-\gamma)}}\|v_h\|_{h,p,\omega_K}^{\gamma}\|v_h\|_{L^q(\omega_K)}^{1-\gamma}\bigg)^s\,.
            \end{aligned}
        \end{align}
        By the definition of $\gamma\in [0,1]$, cf. \eqref{eq:gamma}, it holds that
        \begin{align}\label{eq:gagliardo.4}
            \begin{aligned}
           (1+d(\tfrac{1}{s}-\tfrac{1}{p}))\gamma+d(\tfrac{1}{s}-\tfrac{1}{q})(1-\gamma)=0 \,.
           \end{aligned}
        \end{align}
        Using~\eqref{eq:gagliardo.4}~in~\eqref{eq:gagliardo.3}, in particular, using that each $K\in \Th$ appears only in finitely many $\omega_{K'}$, $K'\in \Th$, we arrive at\enlargethispage{5mm}
        \begin{align}\label{eq:gagliardo.5}
						I_h^2\lesssim \|v_h\|_{h,p}^{\gamma}\|v_h\|_{L^q(\Omega)}^{1-\gamma}\,.
        \end{align}

        Eventually, combining \eqref{eq:gagliardo.2} and \eqref{eq:gagliardo.5} in \eqref{eq:gagliardo.1}, we conclude the assertion.\enlargethispage{2mm}
    \end{proof} 
    {\color{black}
        We should mention that the precise form of the interpolation operator $Q_h$ is not so important, but rather that it satisfies the stability and approximation properties from Lemmas \ref{lem:scott_zhang_stable} and \ref{lem:scott_zhang_approx}. For instance, we could have instead employed the operator from \cite{DR.2007}, which is even defined for Orlicz--Sobolev spaces, but we decided to use the operator from \cite{BO.2009} due to its weaker assumptions on the mesh.
    }

    \section{Parabolic interpolation inequalities for discontinuous elements}\label{sec:parabolic_interpolation}

   \qquad In this section, we derive parabolic interpolation inequalities which will be employed in Section~\ref{sec:stablity} to establish the $L^\infty(I;L^2(\Omega)^d)$-stability of discontinuous Galerkin~schemes.\vspace{-1.5mm}

    \begin{lemma}[Parabolic interpolation inequality]\label{lem:parabolic_interpolation}
			        Let $p,q,s\in [1,\infty)$ be such that $q\leq s$, let $\gamma\in [0,1]$ be such that \eqref{eq:gamma} is satisfied and let $k\in \mathbb{N}_0$. Then, 
        for every $v_h\in L^\infty(I;\P_k(\Th))$, it holds that
        \begin{align*}
				 \textcolor{black}{	\smash{\|v_h\|_{L^r(I;L^s(\Omega))}\lesssim \big\|\|v_h(\cdot)\|_{h,p}\big\|_{L^p(I)}^\gamma  \|v_h\|_{L^\infty(I;L^q(\Omega))}^{1-\gamma}\,,}}
        \end{align*}
        where $r=\smash{\frac{s(p(q+d)-dq)}{(s-q)d}}\in (1,\infty]$.\vspace{-2.5mm} 
    \end{lemma}

    \begin{proof}
			By assumption on $p,q,s\in [1,\infty)$ and $\gamma\in [0,1]$, cf. \eqref{eq:gamma}, we can apply the discrete Gagliardo--Nirenberg-type inequality (cf. Lemma \ref{lem:gagliardo}) to find for almost every $t\in I$ that
        \begin{align}
				 \textcolor{black}{	\smash{\|v_h(t)\|_{L^s(\Omega)}\lesssim \|v_h(t)\|_{h,p}^\gamma\|v_h(t)\|_{L^q(\Omega)}^{1-\gamma}\,,}}\label{eq:parabolic_interpolation}
        \end{align}
				        where $\gamma=\smash{\frac{(s-q)dp}{s(p(q+d)-dq)}}\in [0,1]$. Next, we need to distinguish the cases $s>q$ and $s=q$:
        
        \textit{Case $s>q$.} If $s>q$, then, we have that $0<\gamma\leq 1<p$ and, consequently,~$r=\smash{\frac{p}{\gamma}}\in (1,\infty)$. Raising the inequality \eqref{eq:parabolic_interpolation} to the power $r\in (1,\infty)$, integrating with respect to $t\in I$, pulling out the $L^\infty$-norm of the second factor of the integrand and taking the $r$-th root shows the claim.
        
				\textit{Case $s=q$.} If $s = q$, using Hölder’s inequality, the claim follows with $r = \infty$ and $\gamma = 0$.
    \end{proof}

    \begin{corollary}\label{cor:parabolic_interpolation}
				        Let $p\in [\frac{2d}{d+2},\infty)$ and $k\in \mathbb{N}_0$. Then, 
        for every $v_h\in L^\infty(I;\P_k(\Th))$, it holds that
        \begin{align*}
				\smash{	\|v_h\|_{L^{p_*}(Q)}\lesssim  \big\|\|v_h(\cdot)\|_{h,p}\big\|_{L^p(I)}^\gamma\|v_h\|_{L^\infty(I;L^2(\Omega))}^{1-\gamma}\,,}
        \end{align*}
        where $\gamma=\frac{d}{d+2}$ and $p_*=p\frac{d+2}{d}$.\vspace{-5mm}
    \end{corollary}

    \begin{proof}
			We apply Lemma \ref{lem:parabolic_interpolation} with $q \!= \!2$ and $r \!= \!s\! =\! p_*$, noting that one has admissibility by \eqref{eq:admissibility},~if~${p \!\ge\!  \frac{2d}{d+2}}$. In fact, this is obvious if $p \in [d, \infty)$. For $p \in  [1, d)$, it holds $s=p_* \in [2, p^*]$ if and only if $p\ge  \frac{2d}{d+2}$.
    \end{proof} 

\begin{remark}
Applying the results we have presented so far component-wise, one can obtain analogous statements for vector-valued functions. In this case, one defines the DG~norm~of~$\bw\in W^{1,p}(\Th)^d$ as
    \begin{align*}
    \smash{\|\bw_h\|_{h,p}\coloneqq\Big(\|\nabla_h \bw_h\|_{L^p(\Omega)}^p+\big\|h^{-1/p'}_\Gamma\jump{\bw_h \otimes\bn}\big\|_{L^p(\Gamma_h)}^p\Big)^{\smash{1/p}}}\,.
    \end{align*}
    {\color{black}
    To see this, recall first the norm equivalences in Euclidean space for some $1\leq p\leq q < \infty$:
    \begin{align*}
\left( \sum_{i=1}^d |x_i|^p \right)^{1/p}
\leq d^{\frac{1}{p} - \frac{1}{q}}
\left( \sum_{i=1}^d |x_i|^q \right)^{1/q}\,,
\qquad
\left( \sum_{i=1}^d |x_i|^q \right)^{1/q}
\leq
\left( \sum_{i=1}^d |x_i|^p \right)^{1/p}.
    \end{align*}
Focusing first on the stability property \eqref{eq:scott_zhang_stable}, we see the following using the equivalences above, for arbitrary $\bv \in \mathbb{P}(\Th)^d$:
\begin{align*}
    \|\nabla Q_h \bv\|^p_{L^p(\Omega)}
    &=
    \int_\Omega \left( \sum_{i,j=1}^d |\partial_j Q_h \bv_i|^2 \right)^{p/2}
    \leq \sum_{i=1}^d \int_\Omega \sum_{j=1}^d |\partial_j Q_h \bv_i|^p
    \lesssim
     \sum_{i=1}^d \int_\Omega \left( \sum_{j=1}^d |\partial_j Q_h \bv_i|^2 \right)^{p/2} \\
    &\lesssim
     \sum_{i=1}^d \int_\Omega \left( \sum_{j=1}^d |\partial_{h,j} \bv_i|^2 \right)^{p/2}
     +
     \sum_{i=1}^d \int_\Omega \left( \sum_{j=1}^d |\bv_i^+ \bn_j^+ - \bv_i^- \bn_j^-|^2 \right)^{p/2} 
     \\
    &\lesssim
     \int_\Omega \left( \sum_{i,j=1}^d |\partial_{h,j} \bv_i|^2 \right)^{p/2}
     +
     \int_\Omega \left( \sum_{i,j=1}^d |\bv_i^+ \bn_j^+ - \bv_i^- \bn_j^-|^2 \right)^{p/2}
     =
     \|\bv\|^p_{h,p}\,.
\end{align*}
An analogous argument applies also to the approximation property \eqref{eq:scott_zhang_approx}.
}
\end{remark}

    \begin{remark}
        \textcolor{black}{Denote the broken symmetric gradient of $\bw\in W^{1,p}(\Th)^d$ by $\BD_h\bw \coloneqq  \tfrac{1}{2}(\nabla_h\bw + \nabla_h \bw^\top)$} and
consider the alternative norm, for every $\bw_h \in W^{1,p}(\Th)^d$, defined by
    \begin{align*}
			\smash{{|||}\bm{w}_h{|||}_{h,p} \coloneqq  \Big( \|\BD_h(\bm{w}_h)\|^p_{L^p(\Omega)}+\|h^{-1/p'}_\Gamma\jump{\bm{w}_h\otimes \bm{n}}\|^p_{L^p(\Gamma_h^{i})}
   +\|h^{-1/p'}_\Gamma\bm{w}_h\cdot \bm{n}\|^p_{L^p(\Gamma_h^{\partial})}+\|(\bm{w}_h)_\tau\|^p_{L^p(\Gamma_h^{\partial})} \Big)^{\smash{1/p}}\,,}
    \end{align*}
    where only the normal component $\bw_h \cdot \bn$ is penalised on $\Gamma_h^{\partial}$; here, $(\bw_h)_\tau$ denotes the tangential part of $\bw_h$ on the boundary, i.e., $(\bw_h)_\tau \coloneqq \bw_h - (\bw_h\cdot \bn)\bn$. If one~manages~to~prove~the~existence of a quasi-interpolation operator $Q_h^{\boldsymbol{n}}\colon \P^k(\Th)^d\to W^{1,\infty}(\Omega)^d$ that has analogous stability and approximation properties to those described in Lemma \ref{lem:scott_zhang_stable} and Lemma \ref{lem:scott_zhang_approx}, but using the norm ${|||}\cdot{|||}_{h,p}$, then all the results presented in this work would also apply for the problem with Navier's~slip~boundary~conditions, i.e.,  
    \begin{align*}
    \begin{aligned}
\bu\cdot \bn &= 0 &\quad \text{on }\partial \Omega\,,\\
-(\BS\bn)_\tau &= \gamma \bu_\tau &\quad \text{on }\partial\Omega\,,
 \end{aligned}
    \end{align*}
    where $\gamma>0$ is a parameter. Such a DG method enforces the normal condition $\bu\cdot\bn=0$ weakly, which has been observed to be advantageous in practice;~see,~e.g.,~\cite{GS.2022}. To the best of our knowledge, such an operator is not yet available in the literature.
    \end{remark}

		\section{Stability of DG schemes for non-Newtonian fluids}\label{sec:stablity}
		\subsection{Continuous model and its discretisation}

		\hspace{5mm}Let us assume that the initial data belongs to $\bu_0 \in L^2_{\rmdiv}(\Omega)^d$ and, for simplicity, we will take the forcing function in  $\bm{f}\in C^0(I;L^{p'}(\Omega)^d)$. In the weak formulation of problem \eqref{eq:continuous_PDE}, we look for
		\begin{gather*}
\BS \in L^{p'}(Q)^{d\times d}_{\sym, \tr}\,,\quad
\bu \in L^{p}(I;W^{1,p}_0(\Omega)^d) \cap \Linfty\,, \quad
p \in H^{-1}(I;\Lmean{p'})\,,
		\end{gather*}
		such that for every $\bv\in C^\infty_0(\Omega)^d$, $\phi \in C^\infty_0([0,T))$, and $ q\in C^\infty_0(Q)$, it holds that
		\begin{subequations}\label{eq:weak_PDE}
			\begin{align}
				\BG(\BS, \Du) &= \bm{0} \quad\text{a.e. in }Q\,, \\
				-\int_Q \bu \cdot \bv \partial_t \phi \,\d t\d x
				-\int_\Omega \bu_0 \cdot \bv \phi(0) \,\d x
				+ \int_Q [\BS-\bu \otimes \bu-p\mathbb{I}_d] \fp \BD(\bv) \phi \,\d t\d x
								&= \int_Q \bm{f}\cdot \bv \phi \,\d t\d x\,,\\
				-\int_Q q \rmdiv \bu\,\d t\d x &= 0\,. 
			\end{align}
		\end{subequations}

Note that the exponent $p>1$ is determined by the coercivity condition \eqref{eq:coercivity}. The existence of global weak solutions for large data (assuming $p>\frac{2d}{d+2}$) under monotonicity assumptions for $\BG$ was proved in \cite{BGMS.2012} by working with the graph induced by $\BG$ and also later in \cite{BMM.2021} by working with the function $\BG$ directly. In the non-monotone case, existence of weak solutions is not known, but numerical experiments seem to produce reasonable results \cite{JMPT.2019}.

Let us fix polynomial degrees $k_{\bu}, k_{\pi}\in \Na$ for the  velocity and pressure approximations, respectively; we assume that $k_{\bu}\geq 1$ and $k_{\pi} \leq k_{\bu}$. The spaces corresponding to the discrete approximations~are,~then, defined as
\begin{align*}
\Vh &\coloneqq  \P_{k_{\bu}}(\Th)^d\,,\\
\Mh &\coloneqq  \P_{k_{\pi}}(\Th) \cap \Lmean{p'}\,.
\end{align*}
The space $\Mh$ is equipped with the norm $\norm{\cdot}_{\Lp{p'}}$, while the velocity space $\Vh$~is~equipped~with~the~norm
\begin{align}\label{eq:DG_norm2}
\norm{\cdot}_{h,p} 
\coloneqq \big(
\norm{\BD_h (\cdot)}_{\Lp{p}}^p
+
|\cdot|^p_{\Gamma_h,p}\big)^{1/p}\,,
						\end{align}
			where the \textit{jump semi-norm} for vector-valued functions $\bv_h\in \Vh$ is defined as
			\begin{align}\label{eq:jump_semi-norm2}
|\bv_h|^p_{\Gamma_h,p} \coloneqq
\textcolor{black}{\big\|h^{-1/p'}_\Gamma\jump{\bv_h \otimes\bn}\big\|_{L^p(\Gamma_h)}^p}
=
\int_{\Fh} h_\Gamma^{1-p} |\jump{\bv_h\otimes \bn}|^p\,\d s\,.
			\end{align}
      \textcolor{black}{A useful fact, obtained by combining \eqref{eq:poincare} and \cite[Prop.~2.4]{KR.2022}, is the  \textit{discrete Korn-type inequality}, i.e., for every $\bv_h\in \Vh$, it holds that}
                \begin{align}\label{eq:korn}
                    \begin{gathered}
                        \|\bv_h\|_{L^p(\Omega)}  + \|\nabla_h \bv_h\|_{L^p(\Omega)} \lesssim \|\bv_h\|_{h,p}\,. 
                    \end{gathered}
                \end{align}
Before we present the discretised system, it will be useful to introduce the~notion~of~discrete~gradients. For $l\geq 0$, let us define a \textit{discrete gradient} operator $\mathcal{G}_h^l\colon \Vh \to \P_{\max\{k_{\bu}-1,l\}}(\Th)^{d\times d}$, for every $\bv_h\in \Vh$,  through the relation
\begin{align}\label{eq:discrete_gradient}
\dgrad^l(\bv_h) \coloneqq  \nabla_h \bv_h  - \mathcal{R}^l_h(\bv_h)\quad\text{ in }\P_{\max\{k_{\bu}-1,l\}}(\Th)^{d\times d}\,,
\end{align}
where the \textit{discrete lifting} operator $\mathcal{R}^l_h(\bv_h) \in \P_l(\Th)^{d\times d}$, for every $\bm{t}_h \in \P_l(\Th)^{d\times d}$, is defined through
                \begin{align}\label{eq:lifting_jumps}
                    \int_\Omega \mathcal{R}^l_h(\bv_h)  \fp \bm{t}_h\,\d x
                    =
                     \int_{\Fh} \leftjump \bv_h\otimes \bn \rightjump \fp \avg{\bm{t}_h} \,\d s\,.
                \end{align}
While the natural choice seems to be $l\!=\!k_{\bu}\!-\!1\!\in\! \Na_0$ (this will be set whenever the index~${l\!\in \!\Na_0}$~is~omitted), the number $l\in \Na_0$ is a parameter and can be chosen freely; for instance, if $l=0$, the implementation becomes easier as $\mathcal{R}^l_h$ can be, then, computed through element-wise averages; on the other hand, taking $l=k_{\bu}+1\in\Na$ seems to be advantageous, in the linear case at least, in that the method does not require jump penalisation (cf.\ \cite{JNS.2016}). We will shortly explore yet another choice when defining the discrete convective term. Note that if $\bm{t}_h \in C_0^\infty(\Omega)^{d\times d}$, then this is precisely the distributional gradient of $\bv_h$.\linebreak It is possible to prove stability of the discrete gradient (see, e.g., \cite[Prop.\ 2.1]{dPE.2010} or \cite[Lm.\ 7]{BO.2009}), i.e., that for every $\bv_h \in \Vh$, it holds that
                \begin{align}\label{eq:discrete_gradient_stability}
                    \|\dgrad^l(\bv_h)\|_{L^p(\Omega)} \lesssim \|\bv_h\|_{h,p} \,.
                \end{align}
								The \textit{discrete symmetric gradient} operator $\mathcal{D}^l_h\colon \Vh\to \P_l(\Th)_{\sym}^{d\times d}$, for every $\bv_h\in \Vh$,
        is defined through 
        \begin{align}
            \mathcal{D}^l_h(\bv_h) \coloneqq  \BD_h(\bv_h) - \mathcal{R}^l_{h,\sym}(\bv_h)\quad\text{ in }\P_{\max\{k_{\bu}-1,l\}}(\Th)_{\sym}^{d\times d}\,,
        \end{align} where the \textit{symmetric discrete lifting} operator $ \mathcal{R}_{h,\sym}^l(\bv_h)\in \P_l(\Th)_{\sym}^{d\times d}$, for every $\bm{t}_h \in \P_l(\Th)_{\sym}^{d\times d}$, is defined through
                \begin{align}
                    \int_\Omega \mathcal{R}_{h,\sym}^l(\bv_h)  \fp \bm{t}_h \,\d x
                    =
                     \int_{\Fh} \leftjump \bv_h\otimes \bn \rightjump \fp \avg{\bm{t}_h}\,\d s\,.
                \end{align}
								Similarly, we  define the \textit{discrete divergence} operator $\ddiv^l\colon  \Vh\to \P_{\max\{k_{\bm{u}}-1,l\}}(\Th)$~by~taking~the~trace, i.e.,
        for every $\bv_h\in \Vh$, we define
        \begin{align}
            \ddiv^l(\bv_h) \coloneqq  \tr(\dgrad^l(\bv_h)) = \rmdiv_h(\bv_h) + \tr(\mathcal{R}^l_h(\bv_h))\quad\text{ in }\P_{\max\{k_{\bm{u}}-1,l\}}(\Th)\,.
        \end{align}
        The trace of $\mathcal{R}^l_h(\bv_h)\!\in \!\P_l(\Th)^{d\times d}$  for  $\bv_h\!\in\! \Vh$ can be computed from \eqref{eq:lifting_jumps} by~taking~${\bm{t}_h\! =\! q_h \mathbb{I}_d\!\in\! \P_l(\Th)_{\sym}^{d\times d}}$, where $q_h \in \P_l(\Th)$ is arbitrary and $\mathbb{I}_d\in \mathbb{R}^{d\times d}$ is the identity matrix. In particular, for every $q_h \in \P_l(\Th)$, we can write 
								\begin{align}\label{eq:discrete_divergence}
\int_\Omega q_h \ddiv^l(\bv_h)\,\d x
=
\int_\Omega q_h \rmdiv_h\bv_h\,\d x
- \int_{\Fh} \jump{\bv_h \cdot \bn} \avg{q_h}\,\d s\,.
								\end{align}
								Whenever the index $l\in \Na_0$ is omitted, it is meant that $l= k_{\pi}$, in which case \eqref{eq:discrete_divergence} holds for all $q_h\in \Mh$.

								Regarding the convective term, we wish to preserve the following skew-symmetry~property~that~is valid at the continuous level: for every $\bu,\bv,\bw\in C^\infty_0(\Omega)^d$, where $\rmdiv\bu = 0$ in $\Omega$, it holds that
								\begin{align}
\int_\Omega (\bv\otimes \bu)\fp \nabla \bw \,\d x
= 
-\int_\Omega (\bw\otimes \bu)\fp \nabla \bv\,\d x\,.
								\end{align}
								In the case when discretely divergence-free functions are also point-wise divergence-free (as is, e.g., the case when $\Vh$ is $H(\rmdiv;\Omega)$-conforming and $\Mh = \rmdiv\Vh$),~for~every~${\bu_h,\bv_h,\bw_h\in \Vh}$, we simply define
								\begin{align}\label{eq:convective_term_alternative}
                                    \begin{aligned}
									\hat{\mathcal{C}}_h[\bu_h,\bv_h,\bw_h] &\coloneqq  -\int_\Omega (\bv_h\otimes \bu_h)\fp \dgrad^{2k_{\bu}}(\bw_h)\,\d x
\\&= -\int_\Omega (\bv_h \otimes \bu_h) \fp \nabla_h\bw_h \,\d x
+ \int_{\Fh} \avg{\bv_h \otimes \bu_h}\fp \jump{\bw_h \otimes \bn}\,\d s\,.
\end{aligned}
								\end{align}
								The parameter $2k_{\bu}\in \Na$ in the discrete gradient could be chosen differently, but with this choice one has the second equality, which is straightforward to implement in modern software packages. In general,~we,~then, define the \textit{skew-symmetric convective term} as
								\begin{align}\label{eq:convective_term}
\mathcal{C}_h[\bu_h, \bv_h, \bw_h] \coloneqq
\frac{1}{2}\big[ \hat{\mathcal{C}}_h[\bu_h, \bv_h, \bw_h] 
 - \hat{\mathcal{C}}_h[\bu_h, \bw_h, \bv_h] 
\big].
								\end{align}\pagebreak

Let us now turn our attention towards the time discretisation: we proceed similarly as in \cite{MN.2006,EG21c}. Let $\{\Itau\}_{\tau>0}$ be a family of partitions of the closed time interval $[0,T]$ of the form $ \{I_j\}_{j=1}^{N_\tau}= \{(t_{j-1}, t_j]\}_{j=1}^{N_\tau}$, for some $N_\tau \in \Na$, associated to a (maximal) time step $\tau \coloneqq  \max_{j\in\{1,\ldots ,N_\tau\}} (t_j - t_{j-1})$. We will assume that the family of time partitions is quasi-uniform in the sense that there exists a number $\theta \in (0,1]$ (independent of $\tau>0$) such that\vspace{-1mm}
\begin{align}\label{eq:time_quasiuniform}
\theta \tau \leq \min_{j\in\{1,\ldots, N_\tau\}}(t_j - t_{j-1})\,.
\end{align}
		
We will denote the local space-time cylinders as $Q_j \coloneqq  I_j \times \Omega$ for all $j=1,\ldots, N_\tau$. Then, for a given Banach space $X$ and $k\in \Na_0$, we define the \textit{space of broken (in time) polynomials of global degree $k$ with values in $X$} as\vspace{-1mm}
\begin{align}
\P^k(\Itau;X) \coloneqq  \big\{\bv_\tau\colon [0,T]\to X \mid \bv_\tau|_{I_j}\in \P^k(I_j;X) \text{ for all } j=1,\ldots ,N_\tau \big\}\,.
\end{align}
Note that the functions in $\P^k(\Itau;X)$ are defined at $t=0$ and are left-continuous,~in~particular,~implying that $\bv_\tau(t_j)= \bv_\tau(t_j^-)\coloneqq  \lim_{\smash{s\to t_j^-}} \bv_\tau(s)$  in $X$ at the partition points. For a given function $\bv_\tau\in \P^k(\Itau;X)$, we define the jump at $t_{j-1}$ for every $j\in\{1,\ldots, N_\tau\}$ as
\begin{align}
\begin{aligned}
	\jump{\bv_\tau}_{j-1} &\coloneqq  \bv_\tau(t_{j-1}^+)  - \bv_\tau(t_{j-1})&&\quad \text{ in }X\,,\\
\bv_\tau(t_{j-1}^+) &\coloneqq  \lim_{s\to t_{j-1}^+} \bv_\tau(s)&&\quad \text{ in }X\,.
\end{aligned}
\end{align}

Fix a polynomial degree $k_t\in \Na$ for the time approximation; in the discrete formulation, we will look for a velocity and pressure in the spaces\vspace{-1mm}\enlargethispage{6mm}
\begin{align*}
	\Vht &\coloneqq  \P_{k_t}(\Itau; \Vh)\,,\\
    \Mht &\coloneqq  \P_{k_t}(\Itau; \Mh)\,.
\end{align*}

Now, let $\{\xi_l\}_{l=1}^{k_t+1}$ and $\{\omega_l\}_{l=1}^{k_t +1}$ be the (right-sided) points and weights, respectively, corresponding to the Gauss--Radau quadrature of degree $2k_t\in \Na$ on the reference interval $\hat{I}\coloneqq  (-1,1]$. By applying the transformations $\xi \mapsto \frac{1}{2}(t_{j} + t_{j-1}) + \frac{\xi}{2}(t_j - t_{j-1})$, $\omega \mapsto \frac{\omega}{2}(t_j - t_{j-1})$, one can, then, obtain a quadrature $\{(\xi^j_l,\omega^j_l)\}_{l=1}^{k_t+1}$ on the $I_j$ for all $j\!\in\! \{1,\ldots, N_\tau\}$. This can be used to define the discrete measure $\GRdt$, for every $\boldsymbol{v}_\tau\in C^0(\mathcal{I}_\tau;X)\coloneqq \{\boldsymbol{w}_\tau\in L^\infty(I;X)\mid \boldsymbol{w}_\tau|_{I_j}\in C^0(\overline{I_j};X)\text{ for all }j=1,\dots,N_\tau\}$, where $X$ is a Banach space, as\vspace{-1mm}
\begin{align}\label{eq:discrete_time_measure}
\int_I \boldsymbol{v}_\tau(t)\, \GRdt 
\coloneqq \sum_{j=1}^{N_\tau} \int_{I_j} \boldsymbol{v}_\tau(t)\, \GRdt
\coloneqq  \sum_{j=1}^{N_\tau} \sum_{l=1}^{k_t +1} \omega^j_{l} \boldsymbol{v}_\tau(\xi^j_l)\,.
\end{align}
Here, note the abuse of notation in that we employ the same symbol $\GRdt$ for the integral on all the subintervals $I_j$, $j=1,\ldots, N_\tau$.

We are, eventually, able to introduce the discretisation of \eqref{eq:weak_PDE}. In the discrete formulation, we look for $(\bu_{h,\tau},p_{h,\tau})^\top\in \Vht \times \Mht$ such that for every $(\bv_{h,\tau},q_{h,\tau})^\top\in\Vht\times \Mht$, it holds that
\begin{subequations}\label{eq:discrete_PDE}
\begin{gather}
	\int_Q q_{h,\tau} \ddiv (\bu_{h,\tau})\,\d t\d x + \int_I S^{\pi}_h(p_{h,\tau}; q_{h,\tau}) \,\GRdt = 0\,,
\label{eq:discrete_mass}\\
\sum_{j=1}^{N_\tau}\left[ 
\int_{Q_j} \partial_t \bu_{h,\tau}\cdot \bv_{h,\tau} \,\d t\d x
+ \int_\Omega \jump{\bu_{h,\tau}}_{j-1}\cdot \bv_{h,\tau}(t^+_{j-1})\,\d x
+ \int_{I_j} \mathcal{A}_h(\bu_{h,\tau}; \bv_{h,\tau})\, \GRdt \right.\notag \\
\left.
+ \int_{I_j} \mathcal{C}_h[\bu_{h,\tau},\bu_{h,\tau},\bv_{h,\tau}]\, \GRdt
- \int_{Q_j} p_{h,\tau}\ddiv(\bv_{h,\tau})  \,\d t\d x
\right]
=
\int_Q \bm{f}\cdot \bv_{h,\tau} \,\GRdt\d x\,.
\label{eq:discrete_momentum}
\end{gather}
Here, the initial condition is set as the $L^2$-orthogonal projection into the corresponding discrete space, i.e., $\bu_{h,\tau}(0)\coloneqq\Pi_{\Vh}\bu_0\in \Vh$. The \textit{pressure stabilisation term} above, for every $p_h,q_h \in \Mh$, is defined as
\begin{align}\label{eq:pressure_stabilisation}
	S^{\pi}_h(p_{h}, q_{h}) \coloneqq  \int_{\Gamma_h} h_{\Gamma}^{p'-1} |\jump{p_{h}\bn }|^{p'-2}\jump{p_h \bn} \cdot \jump{q_h\bn}\, \d s\,.
\end{align}
For some $l\in \Na$, the \textit{discretisation of the viscous term}, for every $\bv_h,\bw_h \in \Vh$, is defined as
\begin{align}\label{eq:viscous_term}
	\mathcal{A}_h(\bv_h; \bw_h) \coloneqq  \int_\Omega \hat{\BT}_h \fp \dgrad^l(\bw_h)\,\d x
+ S^{\bu}_h(\bv_h; \bw_h)\,,
\end{align}
where $\hat{\BT}_h\colon \Omega \to \Rds$ is such that
\begin{align}\label{eq:discrete_implicit_relation}
    \BG(\hat{\BT}_h, \hat{\mathcal{G}}_{h,\sym}(\bv_{h})) =  \bm{0} 
\qquad 
\text{in }\Omega\,,
\end{align}
\end{subequations}
where $\hat{\mathcal{G}}_{h,\sym} \in \{\BD_h, \mathcal{G}_{h,\sym}^l\}$. The \textit{velocity stabilisation} for every $\bv_h,\bw_h \in \Vh$, is defined as
\begin{align}\label{eq:velocity_stabilisation}
	S^{\bu}_h(\bv_{h}, \bw_{h}) \coloneqq  \alpha \int_{\Gamma_h} h_{\Gamma}^{1-p} |\jump{\bv_{h}\otimes \bn }|^{p-2}\jump{\bv_h\otimes \bn} \fp \jump{\bw_h \otimes \bn}\, \d s\,,
\end{align}
where $\alpha>0$ is a stabilisation parameter; \textcolor{black}{this term enforces  Sobolev regularity of the velocity vector field as well as the homogeneous Dirichlet boundary condition, as the mesh size tends to zero (cf.~\cite{BE.2008}).}
This choice ensures, thanks to the coercivity condition \eqref{eq:coercivity}, that the discretisation of the viscous term is coercive (in general, for large enough $\alpha>0$), i.e., for every $\bv_h \in \Vh$, it holds that 
\begin{align}\label{eq:discrete_coercivity_A}
\norm{\smash{\hat{\BT}_h}}^{p'}_{\Lp{p'}} +
	\|\bv_h\|^p_{h,p} \lesssim \mathcal{A}_h(\bv_h;\bv_h)\,.
\end{align}

As the discretised system \eqref{eq:discrete_PDE} uses discontinuous polynomials in time, the method~can~be~localised; in practice, the problem is solved on the interval $I_j$ using the information from the (already computed) solution on the previous interval $I_{j-1}$.\medskip

A few additional remarks are in order:\medskip

\noindent 
\textbf{Computing the constitutive relation.} In practice, it is not strictly necessary to compute the function $\smash{\hat{\BS}_{h,\tau}} \colon Q\to \Rds$ corresponding to $\bu_{h,\tau}\in \Vht$ from \eqref{eq:discrete_implicit_relation}. In fact, with modern software tools it is possible to work out the dependence of $\smash{\hat{\BS}_{h,\tau}}$ on $\bu_{h,\tau}$ without~having~to~compute~it~explicitly~(see,~e.g.,~\cite{BH.2021}). For explicit constitutive relations of the type $\BS = \fS(\Du)$, such as \eqref{eq:power_lawA}, this is of course not needed, since one can, then, for every $\bv_h,\bw_h \in \Vh$, write 
\begin{align}\label{eq:viscous_term_explicit}
    \mathcal{A}_h(\bv_h; \bw_h) = \int_\Omega \fS(\hat{\mathcal{G}}_{h,\sym}(\bv_h)) \fp \dgrad^l(\bw_h)\,\d x
+ S^{\bu}_h(\bv_h; \bw_h)\,.
\end{align}
Alternatively, in case a discrete stress is a quantity of interest (or for explicit relations of the type $\Du = \fD(\BS)$ such as \eqref{eq:non_monotone}), one can instead employ a 3-field formulation for the variables $(\BS_{h,\tau},\bu_{h,\tau},p_{h,\tau})^\top$ in the spirit of \cite{FGS.2020}; the results of this work will still hold in that case.\pagebreak

\noindent
\textbf{Various DG methods.} We presented two choices for a discrete gradient in the constitutive relation \eqref{eq:discrete_implicit_relation}. The choice $\smash{\hat{\mathcal{G}}_{h,\sym}}= \mathcal{G}_{h,\sym}^l$, e.g., would lead to a method of Local Discontinuous Galerkin~(LDG)~type. On the other hand, choosing $\smash{\hat{\mathcal{G}}_{h,\sym}} = \BD_h$ leads to an Incomplete Interior Penalty (IIDG) method, which can be advantageous for non-linear problems of the type considered here, since one would not need to explictly compute the lifting terms $\mathcal{R}^l_h(\bu_{h,\tau}),\mathcal{R}^l_h(\bv_{h,\tau})$  in the implementation, thanks to the fact that the full discrete gradient $\mathcal{G}_{h,\sym}^l$ would appear on the test function exclusively (and, therefore, linearly), and so the definition \eqref{eq:lifting_jumps} can be applied directly. Regarding the stabilisation term, one could consider instead 
\begin{align}
	\hat{S}^{\bu}_h(\bv_h; \bw_h) \coloneqq  S^{\bu}_h(\bv_h; \bw_h) - \int_\Omega |\mathcal{R}^l_h(\bv_h)|^{p-2}\mathcal{R}^l_h(\bv_h) \fp \mathcal{R}^l_h(\bw_h)\,\d x
\quad\text{ for all }
\bv_h,\bw_w\in \Vh\,,
\end{align}
which leads to Symmetric Interior Penalty (SIP) methods (cf.\ \cite{MRET.2018}), in the sense that it reduces to the traditional SIP method in the Newtonian case.

\noindent
\textbf{Gauss--Radau Quadrature.} The discrete time measure $\GRdt$ should, in principle, appear in all the time integrals in \eqref{eq:discrete_momentum}; this implies, following the reasoning from \cite{MN.2006,EG21c}, that the method presented here is equivalent to a RadauIIA Runge--Kutta method, which can be readily implemented with many existing software libraries. Note that since the quadrature is exact up to degree $2k_t $, we could omit it from several terms, such as $\int_{Q_j} \partial_t \bu_{h,\tau}\cdot \bv_{h,\tau}\, \GRdt=\int_{Q_j} \partial_t \bu_{h,\tau}\cdot \bv_{h,\tau}\, \d t\d x$.
\textcolor{black}{The RadauIIA method is known for its algebraic stability and its L-stability (or stiff decay and, in particular, A-stability), which makes it suitable for parabolic problems of the type considered in this work (cf.\ \cite{HW.2002,Butcher.2016}).~We~mention,~however, that the arguments presented here can also be applied to other Runge--Kutta methods based on quadrature, such as the LobattoIIIA method (of which Crank-Nicolson is a particular case) and the Gauss--Legendre method (containing the implicit midpoint method). However, the possible semi-implicit nature of some of these methods would lead to CFL-type conditions. For this reason, we focus on the RadauIIA~method exclusively.}

\noindent
\textbf{Divergence constraint and pressure stabilisation.} The motivation behind the pressure stabilisation $S^{\pi}_h$ is the validity of the following inf-sup condition (cf.~\cite[Lm.\ 4.1]{dPE.2010}): 
\begin{align}\label{eq:infsup}
\norm{q_h}_{\Lp{p'}}
\lesssim
\sup_{\bw_h \in \Vh}\frac{\int_\Omega q_h \ddiv(\bw_h)\,\d x}{\norm{\bw_h}_{h,p}}
+ S^{\pi}_h(q_h;q_h)^{\frac{1}{p'}}
\qquad
\text{ for all }\, q_h\in \Mh\,,
\end{align}
whose proof can be found in Appendix \ref{appendix:infsup}. In certain cases, this stabilisation~term~can~be~avoided,~e.g., when matching meshes are used and the pressure is looked for in a continuous subspace (see, e.g., \cite{KR.2022}). Naturally, also for divergence-conforming elements (i.e., when $\Vh \subset H(\rmdiv;\Omega)$ and $\Mh = \rmdiv\Vh$), the stabilisation term $S^{\pi}_h$ is not needed and the divergence constraint \eqref{eq:discrete_mass} simply becomes\enlargethispage{2mm}
\begin{align}
\int_Q q_{h,\tau} \rmdiv\bu_{h,\tau}\,\d x  = 0
\qquad
\text{ for all }\, q_{h,\tau}\in \Mht\,.
\end{align}

\begin{remark}[Method without quadrature]
Sometimes the $\mathrm{dG}(k_t)$ time discretisation method is defined with the usual time integration instead of using the Gauss--Radau quadrature $\GRdt$. In this case, however, the equivalence with a Runge--Kutta method will be lost, in general, \textcolor{black}{and with it the convenience of having available implementations in popular software packages}. That said, the method has also certain nice properties, such as being naturally defined for general $\bm{f}\in L^{p'}(I;W^{-1,p'}(\Omega)^d)$. \textcolor{black}{For instance, for linear problems the only difference is the treatment of the right-hand side, where, e.g., for $\mathrm{dG}(0)$ the method with quadrature employs $\{\bm{f}(t_j)\}_{j\in \{1,\dots,N_\tau\}}$, as quadrature weights,  whereas the method without quadrature leads to the averages $\{\tfrac{1}{t_j-t_{j-1}}\int_{I_j} \bm{f}(t)\,\d t\}_{j\in \{1,\dots,N_\tau\}}$}. All the results in this work also apply to the method without quadrature, with slightly simplified proofs.
\end{remark}

\subsection{A priori estimates and $\Linfty$-stability}

\hspace{5mm}We will now proceed to derive energy estimates for the full discrete problem \eqref{eq:discrete_PDE}.
\textcolor{black}{Incidentally, these a priori estimates can be employed to prove existence of discrete solutions for $p\in (1,\infty)$ through a fixed point argument; see, e.g., \cite{ST.2019} for details.}\vspace{-1mm}\enlargethispage{5mm}

\begin{lemma}[A priori estimates]\label{lem:apriori}
Suppose that $(\bu_{h,\tau}, p_{h,\tau})^\top\in \Vht \times \Mht$ is a solution of problem \eqref{eq:discrete_PDE}, and let $\hat{\BS}_{h,\tau}:Q\to \Rds$ be a function associated to $\bu_{h,\tau}\in \Vht$ in \eqref{eq:discrete_implicit_relation}. Then, assuming the penalty parameter $\alpha>0$ is large enough, there exists a constant $c>0$ (independent of $h,\tau>0$) such that
\begin{align}\label{eq:apriori}
	\begin{split}
\max_{j\in \{1,\ldots, N_\tau\}} \|\smash{\bu_{h,\tau}(t_j)}\|^2_{\Lp{2}}
&+
\sum_{j=1}^{N_\tau} \|\jump{\bu_{h,\tau}}_{j-1}\|^2_{\Lp{2}}
+
\int_I S^{\pi}_h(p_{h,\tau}(t), p_{h,\tau}(t)) \,\GRdt \\
&\quad+
\int_I \|\hat{\BS}_{h,\tau}(t)\|^{p'}_{\Lp{p'}} \,\GRdt
+
\int_I \|\bu_{h,\tau}(t)\|^{p}_{h,p} \,\GRdt
\leq c\,.
\end{split}
\end{align}
For $p=2$, the discrete measure $\GRdt$ can be replaced by the standard measure $\d t$; this is also true for general $p>1$ for the DG method without quadrature.\vspace{-3mm}
\end{lemma}
\begin{proof}
Testing the equations \eqref{eq:discrete_mass} with $q_{h,\tau}\coloneqq p_{h,\tau}\chi_{I_j}\in \Mht$ and \eqref{eq:discrete_momentum}  with $\bv_{h,\tau}\coloneqq \bu_{h,\tau}\chi_{I_j}\in \Vht$  for all $j\in \{1,\dots,N_\tau\}$ and, subsequently, adding the resulting equations, recalling the skew-symmetry property of $\mathcal{C}_h$ (cf. \eqref{eq:convective_term}), for every $j\in \{1,\dots,N_\tau\}$, we find that
\begin{align}\label{eq:apriori.1}
	\begin{aligned}
\frac{1}{2}\int_{I_j} \frac{\d}{\d t}\norm{\smash{\bu_{h,\tau}}}^2_{\Lp{2}}\, \d t
&+
\int_\Omega (\bu_{h,\tau}(t^+_{j-1})- \bu_{h,\tau}(t_{j-1}))\cdot \bu_{h,\tau}(t^+_{j-1})\, \d x
+
\int_{I_j} \mathcal{A}_h(\bu_{h,\tau}; \bu_{h,\tau}) \,\GRdt \\
&\quad+
\int_{I_j} S^{\pi}_h(p_{h,\tau}, p_{h,\tau})\, \GRdt
=
\int_{I_j}\int_{\Omega}{ \bm{f} \cdot \bu_{h,\tau}\,\d x} \,\GRdt\,.
	\end{aligned}\hspace*{-1mm}
\end{align}
Let us assume that the jump penalisation parameter $\alpha>0$ is large enough, so that the coercivity property \eqref{eq:discrete_coercivity_A} is satisfied. Then, using the fact that $2a(a-b)= a^2 - b^2 + (a-b)^2$ for all $a,b\in \mathbb{R}$, together with Hölder’s inequality  from \eqref{eq:apriori.1}, for every $j\in \{1,\dots,N_\tau\}$, we deduce that
\begin{align}\label{eq:apriori.2}
	\begin{aligned}
&\frac{1}{2}\norm{\smash{\bu_{h,\tau}(t_j)}}^2_{\Lp{2}}
- \frac{1}{2}\norm{\smash{\bu_{h,\tau}(t_{j-1})}}^2_{\Lp{2}}
+ \frac{1}{2}\|\jump{\bu_{h,\tau}}_{j-1}\|^2_{\Lp{2}}
+
\int_{I_j} \norm{\smash{\hat{\BS}_{h,\tau}(t)}}^{p'}_{\Lp{p'}} \,\GRdt \\
&\quad 
+ \int_{I_j} \norm{\smash{\bu_{h,\tau}}}^p_{h,p} \GRdt
+ \int_{I_j} S^{\pi}_h (p_{h,\tau}(t); p_{h,\tau}(t)) \,\GRdt \\
&\lesssim 
\bigg(\int_{I_j} \norm{\bm{f}}^{p'}_{\Lp{p'}} \,\GRdt \bigg)^{\smash{1/p'}}
\bigg( \int_{I_j} \norm{\smash{\bu_{h,\tau}}}^p_{\Lp{p}} \,\GRdt \bigg)^{\smash{1/p}}\,.
	\end{aligned}
\end{align}
Applying the $\varepsilon$-Young inequality on the right-hand-side of \eqref{eq:apriori.2}, the discrete Poincar\'e inequality \eqref{eq:poincare}, the discrete Korn inequality \eqref{eq:korn}, and summation with respect to $j\in \{1,\ldots, i\}$, where $i\in\{1, \ldots, N_\tau\}$, for every $i\in \{1, \ldots, N_\tau\}$, for $\varepsilon>0$ small enough, we arrive at
\begin{align}\label{eq:apriori.3}
		\begin{aligned}
&\|\bu_{h,\tau}(t_i)\|^2_{\Lp{2}}
+ \sum_{j=1}^i \frac{1}{2}\|\jump{\bu_{h,\tau}}_{j-1}\|^2_{\Lp{2}}
+ 
\int_0^{t_i} S^{\pi}_h (p_{h,\tau}; p_{h,\tau}) \,\GRdt \\
&\quad +
\int_0^{t_i} \norm{\smash{\hat{\BS}_{h,\tau}}}^{p'}_{\Lp{p'}} \,\GRdt 
+
\int_0^{t_i} \norm{\smash{\bu_{h,\tau}}}^p_{h,p} \,\GRdt 
\lesssim 
\|\bu_0\|^2_{\Lp{2}}
+
\norm{\bm{f}}_{C^0(I;\Lp{p'})}^{p'}\,.
	\end{aligned}
\end{align}
Here, we made use of $\norm{\smash{\bu_{h,\tau}(0)}}_{\Lp{2}} \leq \norm{\bu_0}_{\Lp{2}}$, which is based on the stability of the $L^2$-projection. Taking the maximum with respect to $i\in \{1,\ldots, N_\tau\}$ concludes the proof.
\end{proof}

In the lowest order time discretisation $\mathrm{dG}(0)$, the discrete velocity is piece-wise constant~in~time~and so from the a priori estimate \eqref{eq:apriori} above, one immediately has that (for arbitrary $p>1$)
\begin{align*}
\norm{\smash{\bu_{h,\tau}}}_{\Linfty}
=
\max_{j\in \{1,\ldots, N_\tau\}} \norm{\smash{\bu_{h,\tau}(t_j)}}_{\Lp{2}}
\leq 
c\,.
\end{align*}
One of the main goals of this work is to prove that this bound remains valid for general~polynomial~degree. This will be carried out in the upcoming section.

{\color{black}
\subsubsection{Interlude: the heat equation}
\hspace{5mm}For the sake of clarity, we first present the argument for the derivation of the $L^\infty(I;L^2(\Omega)^d)$-bound in a simplified setting, namely the heat equation and a time discretisation method without quadrature. This type of estimate was obtained for this system for conforming-in-space discretisations in \cite{CW.2006}.
Therefore, in place of \eqref{eq:discrete_PDE}, consider the discrete system seeking for $\bu_{h,\tau}\in \Vht$ such that for every $\bv_{h,\tau}\in \Vht$, it holds that
\begin{align}\label{eq:heat}
	\begin{aligned}
\sum_{j=1}^{N_\tau}\left[ 
\int_{Q_j} \partial_t \bu_{h,\tau}\cdot \bv_{h,\tau} \,\d t\d x
+ \int_\Omega \jump{\bu_{h,\tau}}_{j-1}\cdot \bv_{h,\tau}(t^+_{j-1})\,\d x
+ \int_{Q_j}  \hat{\mathcal{G}}_{h}(\bu_{h,\tau}) \fp \mathcal{G}_h(\bv_{h,\tau})\, \d t\d x \right. \\
\left.
    + \int_{I_j} S_h^{\bu}(\bu_{h,\tau}, \bv_{h,\tau}) \, \d t
\right]
=
\int_Q \bm{f}\cdot \bv_{h,\tau} \,\d t\d x \, . 
\end{aligned}
\end{align}
Observe that, in this case, the a priori estimate \eqref{eq:apriori} becomes:
\begin{align}\label{eq:apriori_heat}
\max_{j\in \{1,\ldots, N_\tau\}} \|\smash{\bu_{h,\tau}(t_j)}\|^2_{\Lp{2}}
+
\sum_{j=1}^{N_\tau} \|\jump{\bu_{h,\tau}}_{j-1}\|^2_{\Lp{2}}
+
\int_I \|\bu_{h,\tau}(t)\|^2_{h,2} \, \d t
\leq c\,.
\end{align}

We will make use of the exponential time interpolant from \cite{CW.2010b}. Fix a parameter $\lambda>0$. For every $j\in\{1,\ldots,N_\tau\}$, we define for polynomials on $I_j$, the linear mapping $\overline{(\cdot)}\coloneqq (r\mapsto \overline{r})\colon \P_{k_t}(I_j) \to \P_{k_t}(I_j)$, for every $r\in \P_{k_t}(I_j)$, through
\begin{subequations}\label{eq:exponential_interpolant_properties}
\begin{align}
\overline{r}(t_{j-1}^+) &= r(t_{j-1}^+)\,, \\
\int_{I_j} \overline{r}(t) q(t) \,\d t 
&=
\int_{I_j} r(t) q(t) e^{-\lambda(t-t_{j-1})} \,\d t
\qquad
\text{ for all } q\in \P_{k_t -1}(I_j)\,.\end{align}
\end{subequations}
Then,~${\overline{(\cdot)}\!\coloneqq \!(\bv_{h,\tau}\!\mapsto\! \overline{\bv}_{h,\tau})\colon \!\P_{k_t}(I_j;\Vh) \!\to \!\P_{k_t}(I_j;\Vh)}$, for every $\bv_{h,\tau}\in\P_{k_t}(I_j;\Vh)$, can be defined through
\begin{align}\label{eq:exponential_interpolant}
	\bv_{h,\tau} =
\sum_{i=0}^k r_i(t) \bv_h^i \in \P_{k_t}(I_j; \Vh)
\mapsto
\overline{\bv}_{h,\tau}=\sum_{i=0}^k \overline{r_i}(t) \bv_h^i \in \P_{k_t}(I_j; \Vh)\,.
\end{align}
One can extend this definition for functions in $\Vht$ in the obvious way. From \cite[Lm.\ 3.6]{CW.2010b} we~know~that if $\norm{\cdot}_{\star}$ is a (semi-)norm on $\Vh$ arising from an (semi-)inner product, then \eqref{eq:exponential_interpolant}~is~\mbox{$L^s(I_j;\Vh)$-stable}, i.e.,
\begin{subequations}\label{eq:exp_stability_star}
\begin{align}
	\bigg(\int_{I_j} \|\overline{\bv}_{h,\tau}(t)\|^s_\star \,\d t \bigg)^{\smash{1/s}}
&	\lesssim
	\bigg(\int_{I_j} \|\bv_{h,\tau}(t)\|^s_\star \,\d t \bigg)^{\smash{1/s}}
&&	\quad\text{ for all } \bv_{h,\tau}\in \P_{k_t}(I_j;\Vh)\,,\; s\in[1,\infty)\,, \label{eq:exp_stability_p_star}\\
	\max_{t\in I_j} \|\overline{\bv}_{h,\tau}(t)\|_\star
&\lesssim
	\max_{t\in I_j} \|\bv_{h,\tau}(t)\|_\star 
&& \quad\text{ for all } \bv_{h,\tau}\in \P_{k_t}(I_j; \Vh)\,. \label{eq:exp_stability_infty_star}
\end{align}
\end{subequations}
In particular, for problem \eqref{eq:heat}, we will apply this result with $\|\cdot\|_\star = \|\cdot\|_{h,2}$.

\begin{proposition}\label{thm:stability_heat}
Suppose that $\bu_{h,\tau}\in \Vht$ is a solution of problem \eqref{eq:heat}. Then, assuming that  $\alpha>0$ is large enough, there exists a constant $c>0$ (independent of $h,\tau>0$) such that
\begin{align}\label{eq:Linfty_stability_heat}
\|\bu_{h,\tau}\|_{\Linfty} \leq c\,.
\end{align}
\end{proposition}\vspace{-3mm}

\begin{proof}
	For $k_t = 0$, the result is a direct consequence of  \eqref{eq:apriori_heat}, so we will only consider the case $k_t>0$.\enlargethispage{7mm}

Fix an arbitrary $j\in \{1,\ldots, N_\tau\}$. We will prove the claim on $L^{\infty}(I_j; L^2(\Omega)^d)$, from which the result \eqref{eq:Linfty_stability} trivially follows. Denote the exponential interpolant of $\bu_{h,\tau}\in \P_{k_t}(I_j;\Vh)$  by $\overline{\bu}_{h,\tau}\in \P_{k_t}(I_j;\Vh)$. Using \eqref{eq:exponential_interpolant_properties}, for every $j\in \{1,\dots,N_\tau\}$, we can examine what happens to the time derivative if we test the problem~\eqref{eq:heat} with $\bv_{h,\tau}\coloneqq\overline{\bu}_{h,\tau}\chi_{I_j}\in \P_{k_t}(I_j;\Vh)$:
\begin{gather}
\int_{Q_j} \partial_t \bu_{h,\tau} \cdot \overline{\bu}_{h,\tau} \,\d t\d x
+ \int_\Omega \jump{\bu_{h,\tau}}_{j-1} \cdot \overline{\bu}_{h,\tau}(t^+_{j-1})\,\d x
= \frac{1}{2} \norm{\smash{\bu_{h,\tau}(t_j)}}^2_{L^2(\Omega)} e^{-\lambda(t_j-t_{j-1})} 
- \frac{1}{2} \norm{\smash{\bu_{h,\tau}(t_{j-1}^+)}}^2_{L^2(\Omega)}\notag\\
+ \frac{\lambda}{2} \int_{I_j} \|\bu_{h,\tau}(t)\|^2_{L^2(\Omega)} e^{-\lambda(t-t_{j-1})} \,\d t
+ \int_\Omega \jump{\bu_{h,\tau}}_{j-1} \cdot \bu_{h,\tau}(t^+_{j-1})\,\d x 
= \frac{1}{2} \norm{\smash{\bu_{h,\tau}(t_j)}}^2_{L^2(\Omega)} e^{-\lambda(t_j-t_{j-1})} \notag\\
+ \frac{1}{2}\|\jump{\bu_{h,\tau}}_{j-1}\|^2_{L^2(\Omega)}
- \frac{1}{2} \norm{\smash{\bu_{h,\tau}(t_{j-1})}}^2_{L^2(\Omega)}
+ \frac{\lambda}{2} \int_{I_j} \|\bu_{h,\tau}(t)\|^2_{L^2(\Omega)} e^{-\lambda(t-t_{j-1})} \,\d t,	\label{eq:Linfty_stability_heat.1}
\end{gather}
where we simply used integration-by-parts in the first term. Noting that $(t\mapsto e^{-\lambda(t-t_{j-1})})\colon \mathbb{R}\to \mathbb{R}_{\ge 0}$, for every $j\in \{1,\dots,N_\tau\}$, is decreasing and dropping positive terms,  from	\eqref{eq:Linfty_stability_heat.1}, for every $j\in \{1,\dots,N_\tau\}$,
we  deduce that
\begin{align}\label{eq:Linfty_stability_heat.2}
	\begin{aligned}
 \frac{\lambda}{2}e^{-\lambda(t_j-t_{j-1})} \int_{I_j} \|\bu_{h,\tau}(t)\|^2_{L^2(\Omega)}  \,\d t
&\leq 
 \frac{1}{2} \norm{\smash{\bu_{h,\tau}(t_{j-1})}}^2_{L^2(\Omega)}
 - \int_{Q_j} \hat{\mathcal{G}}_{h}(\bu_{h,\tau}) \fp \dgrad(\overline{\bu}_{h,\tau}) \,\d t \, \d x \\
&\quad - \int_{I_j} S_h^{\bu}(\bu_{h,\tau}, \overline{\bu}_{h,\tau}) \,\d t
 + \int_{Q_j} \bm{f}\cdot \overline{\bu}_{h,\tau} \,\d t\, \d x\,.
\end{aligned}
\end{align}
The first three terms on the right-hand-side of \eqref{eq:Linfty_stability_heat.2} can be handled immediately as a consequence of the stability properties \eqref{eq:exp_stability_star} and the a priori estimate \eqref{eq:apriori_heat}. Note, in particular, that the jump penalisation term is controlled by the $\|\cdot\|_{h,2}$-norm. As for the last term on the right-hand-side of \eqref{eq:Linfty_stability_heat.2}, from an application of Hölder's inequality, the discrete Poincar\'e inequality \eqref{eq:poincare}, and the stability estimate~\eqref{eq:exp_stability_p_star}, for every $j\in \{1,\dots,N_\tau\}$, it follows that
\begin{align}\label{eq:Linfty_stability_heat.3}
	\begin{aligned}
    \left|\int_{Q_j} \bm{f}\cdot \overline{\bu}_{h,\tau} \,\d t\, \d x \right|
	&\leq 
\bigg(\int_{I_j} \norm{\bm{f}(t)}^{2}_{L^{2}(\Omega)} \,\d t \bigg)^{\smash{1/2}}
\bigg(\int_{I_j} \|\overline{\bu}_{h,\tau}(t)\|^{2}_{L^{2}(\Omega)} \,\d t \bigg)^{\smash{1/2}} \\
									 &\lesssim
\|\bm{f}\|_{C^0(I_j;L^{2}(\Omega)^d)}
\bigg(\int_{I_j} \|\bu_{h,\tau}(t)\|^{2}_{h,2} \,\d t \bigg)^{\smash{1/2}} 
\leq c\,.
\end{aligned}
\end{align}
Putting everything together, using \eqref{eq:Linfty_stability_heat.2} and \eqref{eq:Linfty_stability_heat.3} in \eqref{eq:Linfty_stability_heat.1}, for every $j\in \{1,\dots,N_\tau\}$, we arrive at
\begin{align}\label{eq:almost_stability_heat}
 \frac{\lambda}{2}e^{-\lambda(t_j-t_{j-1})} \int_{I_j} \|\bu_{h,\tau}(t)\|^2_{L^2(\Omega)}  \,\d t
 \leq 
 c\,.
\end{align}
On the other hand, the equivalence of norms in finite dimensional spaces and the quasi-uniformity \eqref{eq:time_quasiuniform} of the time partition imply that (cf. \cite[Lm.\ 3.5]{CW.2010b})
\begin{align}\label{eq:inverse_ineq_time}
\|\bu_{h,\tau}\|^2_{L^\infty(I_j;L^2(\Omega))}
\lesssim
\frac{1}{\tau} \int_{I_j} \norm{\smash{\bu_{h,\tau}(t)}}^2_{L^2(\Omega)} \,\d t\,.
\end{align}
Hence, choosing $\lambda=\tau^{-1}$ in \eqref{eq:almost_stability} yields the claimed a priori estimate.
\end{proof}

\subsubsection{The non-linear problem}

\hspace{5mm}We return now to the non-linear setting. Note that in the definition of the exponential interpolant above, one could use the discrete measure $\GRdt$ as~well,~since~the~Gauss--Radau quadrature integrates exactly up to degree $2k_t$.

Now, as mentioned in the previous section, the stability properties are known to hold whenever the norm $\|\cdot\|_\star$ arises from an inner product which is not the case in the non-linear setting whenever $p\neq 2$. Moreover, when $p\neq 2$, the integrands are not polynomials in time, and so one cannot replace the discrete measure $\GRdt$ with $\d t$.
The following lemma, whose proof can be found in Appendix \ref{appendix:stability}, proves that thankfully these stability properties also hold in the non-linear setting with quadrature-based integrals.

\begin{lemma}\label{lem:exp_stability}
Let $s\in (1,\infty)$ and  $\norm{\cdot}_\star$ is a (semi-)norm on $\Vh$ arising from an (semi-)inner product. Then, the exponential interpolant \eqref{eq:exponential_interpolant}, for every $\bv_{h,\tau}\in \P_{k_t}(I_j;\Vh)$ and $j\in \{1,\ldots, N_\tau\}$, satisfies
\begin{subequations}
\begin{align}
	\bigg(\int_{I_j} \|\overline{\bv}_{h,\tau}(t)\|^s_\star\,  \GRdt \bigg)^{\smash{1/s}}
&	\lesssim
	\bigg(\int_{I_j} \|\bv_{h,\tau}(t)\|^s_\star\,  \GRdt \bigg)^{\smash{1/s}}\,, \label{eq:exp_stability_GR_p_star}\\
	\bigg(\int_{I_j} \|\overline{\bv}_{h,\tau}(t)\|^s_{h,s}\,  \d t \bigg)^{\smash{1/s}}
&	\lesssim
	\bigg(\int_{I_j} \|\bv_{h,\tau}(t)\|^s_{h,s}  \, \d t \bigg)^{\smash{1/s}}\,, \label{eq:exp_stability_p}\\
	\bigg(\int_{I_j} \|\overline{\bv}_{h,\tau}(t)\|^s_{h,s} \, \GRdt \bigg)^{\smash{1/s}}
&	\lesssim
	\bigg(\int_{I_j} \|\bv_{h,\tau}(t)\|^s_{h,s}\, \GRdt \bigg)^{\smash{1/s}}\,. \label{eq:exp_stability_GR_p}
\end{align}
\end{subequations}
\end{lemma}

We are now in a position to prove the sought stability result. We stress now that the restriction on the power-law index $p$ in the stability result arises because the argument requires admissibility in the convective term, but it is conceivable that different arguments could deliver a similar result for the natural range $p>\tfrac{2d}{d+2}$.

\begin{theorem}\label{thm:stability}
Suppose that $(\bu_{h,\tau}, p_{h,\tau})^\top\in \Vht\times \Mht$ is a solution of problem \eqref{eq:discrete_PDE}. Moreover, assume that $p\geq \frac{3d+2}{d+2}$ if $k_t >0$ and $p>1$ if $k_t=0$. Then, assuming that  $\alpha>0$ is large enough, there exists a constant $c>0$ (independent of $h,\tau>0$) such that
\begin{align}\label{eq:Linfty_stability}
\|\bu_{h,\tau}\|_{\Linfty} \leq c\,.
\end{align}
\end{theorem}
\begin{proof}
    Having the stability properties from Lemma \ref{lem:exp_stability} at hand, the proof is almost the same as that of Proposition \ref{thm:stability_heat}. The only essential difference arises from the convective term. Noting first that $p\geq \frac{3d+2}{d+2}$ is equivalent to $2p' \leq p_*$, for every $j\in \{1,\dots,N_\tau\}$, we see that 
\begingroup
\allowdisplaybreaks
\begin{align*}
&    \left| \int_{I_j} \mathcal{C}_h[\bu_{h,\tau},\bu_{h,\tau},\overline{\bu}_{h,\tau}] \,\GRdt\right|
\\&\leq \int_{Q_j}{|\bu_{h,\tau}|^2 |\dgrad^{2k_{\bu}}(\overline{\bu}_{h,\tau})|\,\GRdt\d x}
+ \int_{Q_j} {|\overline{\bu}_{h,\tau}| |\bu_{h,\tau}| |\dgrad^{2k_{\bu}}(\bu_{h,\tau})|\,\GRdt\d x }\\
&\leq
\bigg(\int_{I_j} \|\bu_{h,\tau}(t)\|_{L^{2p'}(\Omega)}^{2p'}\,\GRdt \bigg)^{\smash{1/p'}}
\bigg(\int_{I_j} \|\overline{\bu}_{h,\tau}(t)\|_{h,p}^{p}\,\GRdt \bigg)^{\smash{1/p}} \\
&\quad +
\bigg(\int_{I_j} \|\overline{\bu}_{h,\tau}(t)\|_{L^{2p'}(\Omega)}^{2p'}\,\GRdt \bigg)^{\smash{1/(2p')}}
\bigg(\int_{I_j} \|\bu_{h,\tau}(t)\|_{L^{2p'(\Omega)}}^{2p'}\,\GRdt \bigg)^{\smash{1/(2p')}} \\
&\qquad\times\bigg(\int_{I_j} \|\bu_{h,\tau}(t)\|_{h,p}^{p}\,\GRdt \bigg)^{\smash{1/p}} \\
&\lesssim
\bigg(\int_{I_j} \|\bu_{h,\tau}(t)\|_{L^{p_*}(\Omega)}^{p_*}\,\GRdt \bigg)^{\smash{2/p_*}}
\bigg(\int_{I_j} \|\overline{\bu}_{h,\tau}(t)\|_{h,p}^{p}\,\GRdt \bigg)^{\smash{1/p}} \\
&\quad +
\bigg(\int_{I_j} \|\bu_{h,\tau}(t)\|_{L^{p_*}(\Omega)}^{p_*}\,\GRdt \bigg)^{\smash{1/p_*}}
\bigg(\int_{I_j} \|\overline{\bu}_{h,\tau}(t)\|_{L^{p_*}(\Omega)}^{p_*}\,\GRdt \bigg)^{\smash{1/p_*}} \\
&\qquad\times \bigg(\int_{I_j} \|\bu_{h,\tau}(t)\|_{h,p}^{p}\,\GRdt \bigg)^{\smash{1/p}}\,. 
\end{align*} \endgroup
Note that in the above estimate, we do not need to use the second equality from \eqref{eq:convective_term_alternative} (i.e., the explicit value $2k_{\bu}$ in the discrete gradient is not needed). The important property is \eqref{eq:discrete_gradient_stability}, which holds for $\mathcal{G}_h^{l}(\bu_{h,\tau})$ with any value of $l$.

Now, the crucial observation is that Corollary \ref{cor:parabolic_interpolation} still holds when using the discrete measure $\GRdt$. More precisely,  for every $\bv_{h,\tau} \in \P_{k_t}(I_j; \Vh)$ and $j\in \{1,\dots,N_\tau\}$, we have that
\begin{align}\label{thm:stability.2}
	\begin{aligned}
    &\bigg(\int_{I_j} \|\bv_{h,\tau}(t)\|_{L^{p_*}(\Omega)}^{p_*} \,\GRdt \bigg)^{\smash{1/p_*}}
	\\&\qquad \lesssim \bigg( \int_{I_j} \|\bv_{h,\tau}(t)\|^p_{h,p}\, \GRdt \bigg)^{\smash{1/p_*}}
	\norm{\bv_{h,\tau}}_{L^\infty(I_j;L^2(\Omega)^d)}^{\frac{2}{d+2}}\,.
	\end{aligned}
\end{align}
Combining \eqref{thm:stability.2} with the stability estimate \eqref{eq:exp_stability_infty_star} (with $\norm{\cdot}_\star\! =\! \norm{\cdot}_{L^2(\Omega)}$) and~estimate~\eqref{eq:exp_stability_GR_p}~(with~${s\!=\!p}$), then, yields that
\begin{align}\label{thm:stability.3}
        \left| \int_{I_j} \mathcal{C}_h[\bu_{h,\tau},\bu_{h,\tau},\overline{\bu}_{h,\tau}] \GRdt\right|
\lesssim \|\smash{\bu_{h,\tau}}\|^{\frac{4}{d+2}}_{L^\infty(I_j;L^2(\Omega)^d)}.
\end{align}
On grounds of \eqref{thm:stability.3}, following the same steps as in the proof of Proposition \ref{thm:stability_heat}, for every $j\in \{1,\dots,N_\tau\}$, we arrive at
\begin{align}\label{eq:almost_stability}
 \frac{\lambda}{2}e^{-\lambda(t_j-t_{j-1})} \int_{I_j} \|\bu_{h,\tau}(t)\|^2_{L^2(\Omega)}  \,\d t
 \lesssim 
 1 +
\norm{\smash{\bu_{h,\tau}}}^{\frac{4}{d+2}}_{L^\infty(I_j;L^2(\Omega)^d)}\,.
\end{align}
The argument can be finished in the same way with the inverse estimate \eqref{eq:inverse_ineq_time}, noting that $\frac{4}{d+2}<2$.
\end{proof}
}

\begin{corollary}\label{cor:Linfty_stability2}
Let $(\bu_{h,\tau}, p_{h,\tau})^\top\in \Vht\times \Mht$ be a solution of the discrete problem without quadrature. Moreover, assume that $p\geq \frac{3d+2}{d+2}$ if $k_t >0$ and $p>1$ if $k_t=0$. Then, assuming that  $\alpha>0$ is large enough, there exists a constant $c>0$ (independent of $h,\tau>0$) such that
\begin{align}\label{eq:Linfty_stability2}
\|\bu_{h,\tau}\|_{\Linfty} \leq c\,.
\end{align}
\end{corollary}
\begin{proof}
The proof for the dG time discretisation without quadrature is almost identical.~The~only~difference is that Corollary \ref{cor:parabolic_interpolation} can be applied directly, and that now the stability estimate \eqref{eq:exp_stability_p} with the standard measure $\d t$ is the one that has to be employed.
\end{proof}

{\color{black}
    \section{Convergence of DG schemes for non-Newtonian fluids}\label{sec:convergence}

\hspace{5mm}Having Corollary \ref{cor:Linfty_stability2} at hand, we are now in the position to prove the (weak) convergence of the discrete scheme \eqref{eq:discrete_PDE}. Before we do so, however, let us first derive two basic technical lemmas which will be useful in the sequel:
the first lemma is a weak-* compactness result for sequences of in-time element-wise continuous functions that are bounded in the $L^p$-$L^p$-norm with respect to the discrete~measure~$\GRdt$ in space.

\begin{lemma}[Weak-* convergence of quadrature integrals]\label{lem:weak-star}
	Let $p\in (1,\infty)$ and  $\BT_{\tau_n}\in C^0(\mathcal{I}_{\tau_n};L^p(\Omega))$,~${n\in \mathbb{N}}$,
	a sequence such that
	\begin{align}
		C_{\BT}\coloneqq \sup_{n\in \mathbb{N}}{\bigg(\int_I{\| \BT_{\tau_n}(t)\|_{L^p(\Omega)}^p\,\GRdt}\bigg)^{1/p}}<\infty\,.\label{lem:weak-star.1}
	\end{align} 
	Then, there exists a subsequence $(n_k)_{k\in \mathbb{N}}\subseteq \mathbb{N}$ and a weak limit $\BT\in L^p(I;L^p(\Omega))$ such that for every $\BQ\in C^0(\overline{I};L^{p'}(\Omega))$, it holds that
	\begin{align*}
		\int_I{\int_\Omega{\BT_{\tau_{n_k}}:\BQ\,\mathrm{d}x}\,\GRdt}\to \int_I{\int_\Omega{\BT:\BQ\,\mathrm{d}x}\,\mathrm{d}t}\quad (k\to \infty)\,.
	\end{align*}
\end{lemma}

\begin{proof}

	For every  $\BQ\in C^0(\overline{I};L^{p'}(\Omega))$, by Hölder's inequality and \eqref{lem:weak-star.1}, it holds that
	\begin{align}\label{lem:weak-star.1}
		\begin{aligned}
			\int_I{\int_\Omega{\BT_{\tau_n}:\BQ\,\mathrm{d}x}\,\GRdt}&\leq \bigg(\int_I{\| \BT_n(t)\|_{L^p(\Omega)}^p\,\GRdt}\bigg)^{1/p}\bigg(\int_I{\| \BQ(t)\|_{L^{p'}(\Omega)}^{p'}\,\GRdt}\bigg)^{1/p}\\
			&\lesssim C_{\BT}\,\|\BQ\|_{C^0(\overline{I};L^{p'}(\Omega))}\,,
		\end{aligned}
	\end{align}
  i.e., $(\BT_{\tau_n}\,\GRdt)_{n\in \mathbb{N}}$ is a bounded sequence in the space of Radon measures~$\mathcal{M}(I;L^p(\Omega))$.~Therefore, since $C^0(\overline{I};L^{p'}(\Omega))$ is separable, the Banach--Alaoglu theorem yields a subsequence $(n_k)_{k\in \mathbb{N}}\subseteq \mathbb{N}$ and a weak limit $\boldsymbol{\mu}\in \mathcal{M}(I;L^p(\Omega))$ such that for every $\BQ\in C^0(\overline{I};L^{p'}(\Omega))$, it holds that
	\begin{align}\label{lem:weak-star.2}
		\int_I{\int_\Omega{\BT_{\tau_{n_k}}:\BQ\,\mathrm{d}x}\,\GRdt}\to
		\langle \boldsymbol{\mu},\BQ\rangle_{C^0(\overline{I};L^{p'}(\Omega))} \quad (k\to \infty)\,.
	\end{align}
	Therefore, passing for $k\to \infty$ in \eqref{lem:weak-star.1} with $n=n_k$, for every $\BQ\in C^0(\overline{I};L^{p'}(\Omega))$, we find that
	\begin{align*}
		\langle \boldsymbol{\mu},\BQ\rangle_{C^0(\overline{I};L^{p'}(\Omega))}\lesssim  C_{\BT}\,\bigg(\int_I{\| \BQ(t)\|_{L^{p'}(\Omega)}^{p'}\,\mathrm{d}t}\bigg)^{1/p}\,.
	\end{align*}
	As a consequence, we obtain a function $\BT\in L^p(I;L^p(\Omega))$, such that for every $\BQ\in C^0(\overline{I};L^{p'}(\Omega))$, it holds that
	\begin{align*}
		\langle \boldsymbol{\mu},\BQ\rangle=\int_I{\int_\Omega{\BT:\BQ\,\mathrm{d}x}\,\mathrm{d}t}\,,
	\end{align*}
	which together with \eqref{lem:weak-star.2} yields the assertion.
\end{proof}

The second lemma is a weak-* compactness result for sequences of in-time element-wise continuous functions that are bounded in the $L^p$-$W^{1,p}_0$-norm with respect to the discrete measure $\GRdt$ in space.\enlargethispage{4mm}

\begin{lemma}[Weak-* convergence of quadrature integrals]\label{lem:weak-star_DG}
	Let $p\in [2,\infty)$ and  $\bu_{\tau_n}\in C^0(\mathcal{I}_{\tau_n};W^{1,p}(\Th)^d)\cap L^\infty(I;L^2(\Omega)^d)$, $n\in \mathbb{N}$, a sequence such that
	\begin{align}\label{lem:weak-star_DG.1}
		C_{\bu}\coloneqq \sup_{n\in \mathbb{N}}{\bigg(\int_I{\| \bu_{h_n,\tau_n}(t)\|_{h_n,p}^p\,\GRdt}\bigg)^{1/p}}<\infty\,.
	\end{align} 
	and for some $\bu\in L^\infty(I;L^2(\Omega)^d)$, it holds that
	\begin{align*}
		\bu_{h_n,\tau_n}\overset{\ast}{\rightharpoondown } \bu\quad L^\infty(I;L^2(\Omega)^d)\quad (n\to \infty)\,.
	\end{align*}
	Then, it holds that $\bu\in L^p(I;W^{1,p}_0(\Omega)^d)\cap L^\infty(I;L^2(\Omega)^d)$  and there exists a subsequence $(n_k)_{k\in \mathbb{N}}\subseteq \mathbb{N}$ such that for every $\BQ\in C^0(\overline{I};L^{p'}(\Omega)^{d\times d})$, it holds that
	\begin{align*}
\int_I{\int_\Omega{\mathcal{G}_{h_{n_k}}\bu_{\tau_{n_k}}:\BQ\,\mathrm{d}x}\,\GRdt}\to \int_I{\int_\Omega{\nabla\bu:\BQ\,\mathrm{d}x}\,\mathrm{d}t}\quad (k\to \infty)\,.
	\end{align*}
\end{lemma}

\begin{proof}
	Resorting to Lemma \ref{lem:weak-star}, we find a subsequence $(n_k)_{k\in \mathbb{N}}\subseteq \mathbb{N}$ and a weak limit  $\BG\in L^p(I,L^p(\Omega)^{d\times d})$ such that for every $\BQ\in C^0(\overline{I};L^{p'}(\Omega)^{d\times d})$, it holds that
	\begin{align*}
		\int_I{\int_\Omega{\mathcal{G}_{h_{n_k}}\bu_{h_{n_k},\tau_{n_k}}:\BQ\,\mathrm{d}x}\,\GRdt}\to \int_I{\int_\Omega{\BG:\BQ\,\mathrm{d}x}\,\mathrm{d}t}\quad (k\to \infty)\,.
	\end{align*}
	On the other hand, due to $p\in [2,\infty)$ and \eqref{lem:weak-star_DG.1}, we have that
	\begin{align}\label{lem:weak-star_DG.2}
        \begin{aligned}
		\bigg(\int_I{\|\bu_{h_n,\tau_n}(t)\|_{h_n,2}^2\,\mathrm{d}t}\bigg)^{1/2}&=\bigg(\int_I{\|\bu_{h_n,\tau_n}(t)\|_{h_n,2}^2\,\GRdt}\bigg)^{1/2}\\&\lesssim \bigg(\int_I{\|\bu_{h_n,\tau_n}(t)\|_{h_n,p}^p\,\GRdt}\bigg)^{1/p}\leq c\,.
        \end{aligned}
	\end{align} 
	Therefore, proceeding as in \cite[Prop. 4.27]{KR.2023}, we find that $\bu\in L^2(I;W^{1,2}_0(\Omega)^d)$ as well as
	\begin{align*}
		\mathcal{G}_{h_n}\bu_{h_n,\tau_n}\rightharpoonup \nabla \bu\quad\text{ in }L^2(I;L^2(\Omega)^{d\times d})\quad (n\to \infty)\,.
	\end{align*}
	Denoting, for each $n\in \mathbb{N}$, by $I_{\tau_n}\colon C^0(\mathcal{I}_{\tau_n};X)\to C^0(\overline{I};X)$, where $X$ is a Banach~space,~the~Lagrange interpolation operator associated with the Gauss--Radau nodes from $\mathcal{I}_{\tau_n}$, for every $\BQ\in C^\infty_0(Q)^{d\times d}$, we find that
	\begin{align*}\int_I{\int_\Omega{\mathcal{G}_{h_n}\bu_{h_n,\tau_n}:\BQ\,\mathrm{d}x}\,\GRdt}&=\int_I{\int_\Omega{\mathcal{G}_{h_n}\bu_{h_n,\tau_n}:I_{\tau_n}(\BQ)\,\mathrm{d}x}\,\mathrm{d}t}\\&\to \int_I{\int_\Omega{\nabla\bu:\BQ\,\mathrm{d}x}\,\mathrm{d}t}\quad (n\to \infty)\,,
	\end{align*}
	so that for every $\BQ\in C^\infty_0(Q)^{d\times d}$, we obtain
	\begin{align*}
		\int_I{\int_\Omega{\nabla\bu:\BQ\,\mathrm{d}x}\,\mathrm{d}t}=\int_I{\int_\Omega{\BG:\BQ\,\mathrm{d}x}\,\mathrm{d}t}\,.
	\end{align*}
	The fundamental theorem in the calculus of variations yields that $\nabla\bu=\BG\in L^p(I;L^p(\Omega)^{d\times d})$, so that, using Poincar\'e's inequality, we conclude that $\bu\in L^p(I;W^{1,p}_0(\Omega)^d)$.
\end{proof}

We are, eventually, in the position to prove a convergence result. For the sake of convenience, we treat the case with a constitutive relation of the type $\BS = \fS(\BD(\bu))$, but analogous arguments can be used to treat the general case.\vspace{-1mm}\enlargethispage{1mm}

\begin{theorem}\label{thm:convergence}
     Let $(\bu_{h,\tau}, p_{h,\tau})^\top\in \Vht\times \Mht$ be a solution of the discrete problem. Moreover, assume that $p\geq \frac{3d+2}{d+2}$. Then, assuming that  $\alpha>0$ is large enough, there exist null sequences $(h_n)_{n\in \mathbb{N}},(\tau_n)_{n\in \mathbb{N}}\subseteq (0,1)$ and a vector field $\bu\in L^p(I;W^{1,p}_{0,\mathrm{div}}(\Omega))\cap  L^\infty(I;L^2(\Omega)^d)$ such that
	\begin{align*}
		\begin{aligned}
		\bu_{h_n,\tau_n}&\overset{\ast}{\rightharpoondown} \bu&&\quad \text{ in }L^\infty(I;L^2(\Omega)^d)&&\quad (n\to \infty)\,,\\
			\mathcal{G}_{h_n}\bu_{h_n,\tau_n}&\rightharpoonup \nabla\bu&&\quad \text{ in }L^2(I;L^2(\Omega)^{d\times d})&&\quad (n\to \infty)\,,\\
		\bu_{h_n,\tau_n}&\to \bu&&\quad \text{ in }L^2(I;L^2(\Omega)^d)&&\quad (n\to \infty)\,,
		\end{aligned}
	\end{align*}
	Furthermore, it follows that $\bu\in W^{1,p,p'}(I;W^{1,p}_{0,\mathrm{div}}(\Omega),W^{1,p}_{0,\mathrm{div}}(\Omega)^*)\hookrightarrow C^0(\overline{I};L^2_{0,\mathrm{div}}(\Omega))$ satisfies \eqref{eq:weak_PDE}.\vspace{-1mm}
\end{theorem}

\begin{proof}
    \textit{1. Convergences:} From Lemma \ref{lem:apriori} together with Lemma \ref{lem:weak-star} and Lemma \ref{lem:weak-star_DG} as well as Corollary \ref{cor:Linfty_stability2} together with the Banach--Alaoglu theorem as $L^\infty(I;L^2(\Omega)^d)$ has a separable pre-dual, we deduce the existence of sequences $(h_n)_{n\in \mathbb{N}},(\tau_n)_{n\in \mathbb{N}}\subseteq (0,1)$ and of a vector field $\bu\in L^p(I;W^{1,p}_{0}(\Omega)^d)\cap  L^\infty(I;L^2(\Omega)^d)$ as well as of a tensor field $\BS\in L^{p'}(I;L^{p'}(\Omega)^{d\times d})$ such that 
    \begin{align}\label{eq:velocity_convergences}
		\begin{aligned}
			\bu_{h_n,\tau_n}&\overset{\ast}{\rightharpoondown} \bu&&\quad \text{ in }L^\infty(I;L^2(\Omega)^d)&&\quad (n\to \infty)\,,\\
			\mathcal{G}_{h_n}\bu_{h_n,\tau_n}&\rightharpoonup \nabla\bu&&\quad \text{ in }L^2(I;L^2(\Omega)^{d\times d})&&\quad (n\to \infty)\,,\\
		\hat{\BS}_{h_h,\tau_h}&\rightharpoonup\BS&&\quad \text{ in }L^{p'}(I;L^{p'}(\Omega)^{d\times d})&&\quad (n\to \infty)\,.
		\end{aligned}
	\end{align}

  In order to see that, in fact, $\bu\in L^p(I;W^{1,p}_{0}(\Omega)^d)\cap  L^\infty(I;L^2(\Omega)^d)$ satisfies the divergence constraint, take $q\hspace*{-0.1em}\in\hspace*{-0.1em} C^\infty_c(Q)$ and denote its projection by $q_{h_n} \hspace*{-0.1em}\coloneqq  \hspace*{-0.1em}\Pi_{\smash{\mathbb{M}^{h_n}}} q\hspace*{-0.1em}\in\hspace*{-0.1em}\mathbb{M}^{h_n}$.
  Writing~${{(h_n)_\Gamma}^{p'-1} \hspace*{-0.1em}= \hspace*{-0.1em}{(h_n)_\Gamma}^{(p'-1)/p}{(h_n)_\Gamma}^{1/p}}$ for all $n\in \mathbb{N}$ and using Hölder's inequality, for every $n\in \mathbb{N}$, one can bound the pressure stabilisation term as\vspace{-1mm}
  \begin{align}\label{thm:convergence.1}
\int_I      \mathcal{S}_{h_n}^{\pi}(p_{h_n};q_{h_n})\, \GRdt
\leq
\left(\int_I{(h_n)_\Gamma^{p'-1} |\leftjump p_{h_n} \bn \rightjump|^{p'} \, \d s }\right)^{1/p}
\left(\int_I {(h_n)_\Gamma^{\frac{1}{p-1}} |\leftjump q_{h_n} \bn \rightjump|^{p'} \,\d s}\right)^{1/p'}.
  \end{align}
  Note that in  \eqref{thm:convergence.1} the facet size function $(h_n)_\Gamma\colon \Gamma_{h_n}\to \mathbb{R}_{>0}$ only appears with positive exponents. Therefore, using the a priori estimate \eqref{eq:apriori},
  the convergence properties \eqref{eq:velocity_convergences}, and the approximation properties of the discrete projection $\Pi_{\mathbb{M}^{h_n}}\colon L^{p'}(\Omega)\to \mathbb{M}^{h_n}$ it is possible to pass to the limit in the discrete divergence constraint \eqref{eq:discrete_mass} and conclude that for every $q\in C_0^\infty(Q)$,~it~holds~that $\int_Q \rmdiv \bu \, q\,\d t\d x = 0$.

  Now, noting that Corollary \ref{cor:parabolic_interpolation} remains valid when using the discrete measure $\GRdt$, we obtain
  \begin{align}\label{eq:bound_L2p}
      \int_I \|\bu_{h_n,\tau_n}\|^{2p'}_{L^{2p'}(\Omega)} \GRdt \leq c,
  \end{align}
  where $c>0$ is independent of $n\in \mathbb{N}$. In \eqref{eq:bound_L2p}, we used that $p\geq \smash{\frac{3d+2}{d+2}}$, which implies~that~$2p' \leq p\smash{\frac{d+2}{d}}$. Therefore, if we take a discrete approximation $\bv_{h_n,\tau_n}$ of an arbitrary $\bv\in C_c^\infty(Q)^d$, an application of Lemma \ref{lem:weak-star} then ensures the existence of functions $\BH_1 \in L^{p'}(Q)^{d\times d}$ and $\BH_2 \in L^{(2p')'}(Q)^{d\times d}$ such that
\begin{align*}
\int_I{\mathscr{C}_h[\bu_{h_n,\tau_n},\bu_{h_n,\tau_n},\bv_{h_n,\tau_n}]\, \GRdt}
    \to
    \int_Q \BH_1 \fp \nabla \bv \, \d x\, \d t
    + \int_Q \BH_2\cdot \bv \, \d x \,\d t\quad (n\to \infty)\,. 
	\end{align*}\newpage
	
	\textit{2. Compactness of the velocity vector field:} In order to prove compactness of the sequence of discrete velocity vector fields, we consider the space of discretely divergence-free velocities, for every $h>0$, defined by $$ \Vhdiv\coloneqq  \bigg\{\bw^h \in \Vh \,\Big|\, \int_\Omega{ q^h \rmdiv \bw^h\,\d x} = 0 \text{ for all }q^h \in \Mh\bigg\}\,,$$
	and equip it with either of the two norms $\smash{\|\cdot\|_{\Vhdiv}, \|\cdot\|_{(\Vhdiv)^*} \colon \Vhdiv\to \mathbb{R}_{\ge 0}}$, for every $\smash{\bv^h\in \Vhdiv}$,~defined~by
	\begin{align*}
		\|\bv^h\|_{\smash{\Vhdiv}} \coloneqq  \|\bv^h\|_{h,p} \,,
		\qquad
		\|\bv^h\|_{\smash{(\Vhdiv)^*}} \coloneqq  \smash{\sup_{\bw^h \in \Vhdiv\setminus\{0\}}}\frac{\int_\Omega \bv^h\cdot \bw^h\,\d x}{\|\bw^h\|_{\smash{\Vhdiv}}}\,.
	\end{align*}
	
	Noting that, thanks to the a priori bounds (cf.\ Lemma \ref{lem:apriori}), that $p\in [2,\infty)$, and similar to  \eqref{lem:weak-star_DG.2}, we have that $\sup_{n\in \mathbb{N}}{\|\bu_{h_n,\tau_n}\|_{L^2(0,T;W^{1,2}(\mathcal{T}_{h_n}))}}<\infty$
	and due to the spatial compactness (i.e., bounded sequences in $W^{1,2}(\mathcal{T}_{h_n})$ are pre-compact in $L^2(\Omega)^d$ (cf. \cite[Thm.\ 5.2]{BO.2009})), the only condition left to check in order to obtain compactness is that (see \cite[Prop.\ 9]{CJL.2014} and \cite[Rmk.\ 6]{GL.2012})
	\begin{align*}
		\int_0^{T-\delta}\|\bu_{h_n,\tau_n}(t+\delta) - \bu_{h_n,\tau_n}(t)\|_{(\mathbb{V}_{\textup{div}}^{h_n})^*}\,\d t \to 0\quad(\delta\to 0) 
		\quad
		\text{uniformly in }n\in \mathbb{N}\,.
	\end{align*}
	We will now argue similarly as in \cite{W.2010} and use discrete approximations of characteristic functions~in~time. Fix $j\in\{1,\ldots, N_{\tau_n}\}$ and let $\tilde{p}^j(\cdot;t)\in \mathbb{P}^{k_t}(I_j)$, for arbitrary $t\in I_j$, be such that $	\tilde{p}^j(t^+_{j-1};\cdot) = 1$~and~such~that for every $	q\in \mathbb{P}^{k_t-1}(I_j)$, it holds that\enlargethispage{5mm}
	\begin{align*}
		\int_{t_{j-1}}^{t_j} \tilde{p}^j(s;t)q(s)\,\mathrm{d}s 
		=
		\int_{t_{j-1}}^{t} q(s)\,\mathrm{d}s \,.
	\end{align*}
	An explicit construction can be found in \cite{W.2010}. Moreover, one has that
	\begin{alignat}{2}
		\|\tilde{p}^j(\cdot,t)\|_{L^\infty(I_j)} &\leq c(k_t)
	&&\quad
		\text{ for all }t\in I_j\,, \label{eq:discrete_characteristic1} \\
		\|\tilde{p}^j(\cdot, s+\delta) - \tilde{p}^j(\cdot;s)\|_{L^\infty(I_j)} 
		&\leq 
		c(k_t)\,\delta\tau^{-1}
		&&\quad
		\text{ for all }s\in I_j\text{ with }s+\delta \in I_j\,. \label{eq:discrete_characteristic2}
	\end{alignat}
	An important consequence of the definition of $\tilde{p}^j(\cdot;t)\in \mathbb{P}^{k_t}(I_j)$ is that, testing \eqref{eq:discrete_PDE}, for every $n\in \mathbb{N}$, with $\bv^{h_n,\tau_n} = \tilde{p}(\cdot;t)\bw^{h_n}\in \mathbb{V}^{h_n,\tau_n}$, where $\bw^{h_n}\in \mathbb{V}^{h_n}_{\textup{div}}$, for every $n\in \mathbb{N}$, one obtains
	\begin{align}\label{eq:discrete_characteristic}
		\int_\Omega (\bu_{h_n,\tau_n}(t) - \bu_{h_n,\tau_n}(t_{j-1}))\cdot \bw^{h_n}
		=
		\int_{t_{j-1}}^{t_j}  \langle F^{h_n}_{\bu}(s),  \bw^{h_n}\rangle_{\mathbb{V}^{h_n}_{\textup{div}}} \tilde{p}^j(s;t) \,\mu^{\mathrm{GR}}_{k_t+1}(\mathrm{d}s)\,,
	\end{align}
	where  $F^{h_n}_{\bu}(s)\in (\mathbb{V}^{h_n}_{\textup{div}})^*$ for a.e.\ $s\in I_j$, $j\in\{1,\ldots, N_{\tau_n}\}$, and $n\in \mathbb{N}$ denotes the spatial momentum residual.
	
	Now, take $\delta >0$ and suppose, without loss of generality, that $\delta < \theta\tau$. We consider first the case in which $s\in (t_{j-1}, t_j)$ and $s+\delta \in (t_{j-1},t_j]$. Then, from \eqref{eq:discrete_characteristic}, for every $n\in \mathbb{N}$, we get that
	\begin{align}
        \|\bu_{h_n,\tau_n}(s+\delta) - \bu_{h_n,\tau_n}(s)\|_{(\mathbb{V}_{\textup{div}}^{h_n})^*} &\leq  
		\left(\int_{I_j} \|F^{h_n}_{\bu}(t)\|_{( \mathbb{V}^{h_n}_{\textup{div}})^*}\,\mu^{\mathrm{GR}}_{k_t+1}(\mathrm{d}t) \right) 
		\|\tilde{p}^j(\cdot; s+ \delta) - \tilde{p}^j(\cdot;s)\|_{L^\infty(I_j)} \notag\\
		& \leq c\,
		\|\tilde{p}^j(\cdot; s+ \delta) - \tilde{p}^j(\cdot;s)\|_{L^\infty(I_j)}\,,\label{eq:discrete_characteristic.2}
	\end{align}
	where used \eqref{eq:discrete_characteristic}, the a priori estimate in Lemma \ref{lem:apriori}, and that  $p\geq \frac{3d+2}{d+2}$. 
    Integration of \eqref{eq:discrete_characteristic.2} with respect to $s\in (t_{j-1},t_j - \delta)$ for all $j\in \{1,\dots,N_{\tau_n}\}$ and 
     summation with respect to $j\in \{1,\dots,N_{\tau_n}\}$~yields~that
	\begin{align}\label{eq:shifts_bound1}
		\sum_{j=1}^{N_{\tau_n}}\int_{t_{j-1}}^{t_j - \delta} \|\bu_{h_n,\tau_n}(s+\delta) - \bu_{h_n,\tau_n}(s)\|_{( \mathbb{V}^{h_n}_{\textup{div}})^*} \,\mathrm{d}s
		\lesssim \delta\,.
	\end{align}
	
	Now, we consider the case when $s\in (t_{j-1}, t_j)$ and $s+\delta \in (t_j, t_{j+1})$. Note that $s\in (t_j-\delta, t_j)$, so, in particular, $j\leq N_{\tau_n} -1$. Then, from \eqref{eq:discrete_characteristic}, for every $n\in \mathbb{N}$ and arbitrary $\bw^{h_n}\in \mathbb{V}^{h_n}_{\textup{div}}$, we have that\enlargethispage{5mm}
	\begin{align*}
		\int_\Omega (\bu_{h_n,\tau_n}(s+\delta) - \bu_{h_n,\tau_n}(s))\cdot \bw^{h_n}\,\d x &=
		\int_{t_j}^{t_{j+1}} \langle F^{h_n}_{\bu}(\eta), \bw^{h_n} \rangle \tilde{p}^{j+1}(\eta; s+ \delta)\, \mu^{\mathrm{GR}}_{k_t+1}(\mathrm{d}\eta)
		\\&\quad+
		\int_{t_{j-1}}^{t_{j}} \langle F^{h_n}_{\bu}(\eta), \bw^{h_n} \rangle (1 - \tilde{p}^{j}(\eta; s))\, \mu^{\mathrm{GR}}_{k_t+1}(\mathrm{d}\eta) \,.
	\end{align*}
	Using now the bound \eqref{eq:discrete_characteristic1}, integration with respect to $s\in (t_j-\delta, t_j)$ and summation with respect to $j\in \{1,\ldots, N_{\tau_n} -1\}$, for every $n\in \mathbb{N}$, we obtain
	\begin{align}\label{eq:shifts_bound2}
		\sum_{j=1}^{N_{\tau_n}-1}\int_{t_{j}-\delta}^{t_j} \|\bu_{h_n,\tau_n}(s+\delta) - \bu_{h_n,\tau_n}(s)\|_{(\mathbb{V}^{h_n}_{\textup{div}})^*} \,\mathrm{d}s \lesssim \delta\,.
	\end{align}
	Combining \eqref{eq:shifts_bound1} with \eqref{eq:shifts_bound2} yields the required compactness in $L^2(I;L^2(\Omega)^d)$. In particular, we obtain
	\begin{align} \label{eq:strong_L2_convergence}
		\bu_{h_n,\tau_n}\to \bu \quad\text{ in }L^2(I;L^2(\Omega)^d)\quad (n\to \infty)\,.
	\end{align} 
	
	\textit{3. Convergence of the time derivative:}
	From integration-by-parts, for every $n\in \mathbb{N}$, it follows that
	\begin{align}\label{eq:something}
		\begin{aligned}
		\sum_{j=1}^{N_{\tau_n}} &\left[ 
		\int_{Q_j} \partial_t \bu_{h_n,\tau_n}\cdot \bv_{h_n,\tau_n} \,\d x\d t
		+ \int_\Omega \jump{\bu_{h_n,\tau_n}}_{j-1}\cdot \bv_{h_n,\tau_n}(t^+_{j-1})\,\d x
		\right]
		\\&=
		\sum_{j=1}^{N_{\tau_n}}\left[ 
		\int_{Q_j} \bu_{h_n,\tau_n}\cdot \partial_t \bv_{h_n,\tau_n} \,\d x\d t
		+ \int_\Omega \bu_{h_n,\tau_n}(t_j) \cdot \bv_{h_n,\tau_n}(t_j) - \bu_{h_n,\tau_n}(t_{j-1})\cdot \bv_{h_n,\tau_n}(t^+_{j-1}) \,\d x
		\right]\,.
	\end{aligned}
	\end{align}
	Now, choose $\bv_{h_n,\tau_n}= I_{\tau_n}(\phi)\Pi_{\mathbb{V}^{h_n}}\bv\in \mathbb{V}^{h_n,\tau_n}$, for arbitrary $\phi \in C_c^\infty([0,T))$ and $\bv\in C_c^\infty(\Omega)^d$.\enlargethispage{5mm}
Then, from \eqref{eq:something}, it follows that
	\begin{align}\label{eq:something2}
		\begin{aligned}
		\int_{Q_j}& \partial_t \bu_{h_n,\tau_n}\cdot I_{\tau_n}(\phi)\Pi_{\mathbb{V}^{h_n}}\bv \,\d x\d t
		+ \sum_{j=1}^{N_{\tau_n}}
		 \int_\Omega \jump{\bu_{h_n,\tau_n}}_{j-1}\cdot I_{\tau_n}(\phi)(t^+_{j-1}) \Pi_{\mathbb{V}^{h_n}}\bv\,\d x
		\\&=
		\int_{Q} \bu_{h_n,\tau_n}\cdot \partial_t (I_{\tau_n}(\phi)) \Pi_{\mathbb{V}^{h_n}}\bv \,\d x\d t 
		- \int_\Omega \bu_{h_n,\tau_n}(0)\cdot \phi(0) \Pi_{\mathbb{V}^{h_n}}\bv \,\d x
		 \\
		& \quad+ \sum_{j=1}^{N_{\tau_n}}
		\int_\Omega \left[ \bu_{h_n,\tau_n}(t_j) \cdot I_{\tau_n}(\phi)(t_j)\Pi_{\mathbb{V}^{h_n}}\bv - \bu_{h_n,\tau_n}(t_{j-1})\cdot I_{\tau_n}(\phi)(t^+_{j-1})\Pi_{\mathbb{V}^{h_n}}\bv 
		\right] \,\d x\,,
	\end{aligned}
	\end{align}
	where we used that $I_{\tau_n}(\phi)(0)=\phi(0)$ and $I_{\tau_n}(\phi)(T)=0$. The first two terms in \eqref{eq:something2} can be handled immediately by recalling that $\bu_{h_n,\tau_n}\to \bu$ in $L^2(I;L^2(\Omega)^d)$ $(n\to \infty)$, $\partial_t(I_{\tau_n} \phi) \to \partial_t\phi$ in $L^2(I)$ $(n\to \infty)$, $\Pi_{\mathbb{V}^{h_n}}\bv\to \bv$ in $L^2(\Omega)^d$ $(n\to \infty)$, and $\bu_{h_n,\tau_n}(0)\to \bu_0$ in $L^2(\Omega)^d$ $(n\to \infty)$. For the last term, adding and subtracting the term $\bu_{h_n,\tau_n}(t_{j-1})\cdot I_{\tau_n}(\phi)(t_{j-1})\Pi_{\mathbb{V}^{h_n}}\bv $, noting that $|I_{\tau_n} \phi(t_{j-1}) - I_{\tau_n}\phi(t_{j-1}^+)|$ and employing the uniform $L^\infty(0,T;L^2(\Omega)^d)$-bound, we can conclude that
	\begin{align*}
		\sum_{j=1}^{N_{\tau_n}} &\left[ 
		\int_{Q_j} \partial_t \bu_{h_n,\tau_n}\cdot \bv_{h_n,\tau_n} \,\d x\d t
		+ \int_\Omega \jump{\bu_{h_n,\tau_n}}_{j-1}\cdot \bv_{h_n,\tau_n}(t^+_{j-1})\,\d x
		\right]\\&\quad \to
		-\int_Q \bu \cdot \partial_t\phi \bv \,\d x\,\d t - \int_\Omega \bu_0 \cdot \phi(0)\bv \,\d x\quad (n\to \infty)\,.
	\end{align*} 

	\textit{4. Identifying the initial condition:}
  So far we have proved that by taking $\bv_{h_n,\tau_n} = I_{\tau_n}(\phi)\Pi_{\Vh}\bv$  in \eqref{eq:discrete_momentum}, for arbitrary $\phi \in C_c^\infty([0,T))$, $\bv\in C^\infty_{c,\rmdiv}(\Omega)^d$, and $n\in \mathbb{N}$, by passing for $n\to \infty$, we obtain
  \begin{align}\label{eq:equation}
  \begin{aligned}
-\int_Q \bu\cdot \partial_t\phi\, \bv \d x\, \d t&
+ \int_Q \BS \fp \phi \BD \bv \, \d x \, \d t
    + \int_Q \BH_1 \fp \phi \nabla \bv \, \d x\, \d t
    + \int_Q \BH_2\cdot \phi \bv \, \d x \,\d t
   \\& =\int_\Omega \bu_0 \cdot \phi(0)\bv \, \d x+ \int_Q \bm{f}\cdot \phi \bv \, \d x \,\d t\,.
   \end{aligned}
  \end{align}
  From \eqref{eq:equation}, it follows that $\partial_t \bu \in L^{q}(I;X_{\rmdiv}(\Omega)^*)$, where $q\coloneqq(\max\{2p',p\})'$ and
   $X_{\rmdiv}(\Omega)$ is the closure of $C_{c,\rmdiv}^\infty(\Omega)^d$ in   $W^{1,p}_{0,\rmdiv}(\Omega)^d \cap L^{2p'}(\Omega)^d$, i.e., with respect to the norm $\|\cdot\|_{X} \coloneqq  \|\cdot\|_{W^{1,p}(\Omega)} + \|\cdot\|_{L^{2p'}(\Omega)}$.
  Since the embeddings $X_{\rmdiv}(\Omega) \hookrightarrow L^2_{\rmdiv}(\Omega)^d \hookrightarrow X_{\rmdiv}(\Omega)$ are continuous and dense, $\partial_t \bu \in L^{q}(I;X_{\rmdiv}(\Omega)^*)$, and $\bu\in  L^\infty(I;L^2_{\rmdiv}(\Omega))$, we have that $\bu \in C_w(I;L^2_{\rmdiv}(\Omega)^d)$. Thus, standard arguments (see, e.g., \cite{ST.2019})  yield that $\bu(0) = \bu_0$ in $L^2_{\rmdiv}(\Omega)^d$ and,  in fact, 
  \begin{align}\label{eq:conv_in_0}
      \lim_{t\to 0^+} \|\bu(t) - \bu_0\|_{L^2(\Omega)} = 0\,.
  \end{align}
  In addition, standard arguments (cf. \cite{BKR.2021,KRUnsteady.2023}) show that $ \|\bu(T)\|_{L^2(\Omega)}^2\leq \liminf_{n\to\infty}{\|\bu_{h_n,\tau_n}(T)\|_{L^2(\Omega)}^2}$ or, eqivalently,
  \begin{align}\label{eq:conv_in_T}
  \limsup_{n\to\infty}{\big[-\|\bu_{h_n,\tau_n}(T)\|_{L^2(\Omega)}^2\big]}\leq -\|\bu(T)\|_{L^2(\Omega)}^2\,.
  \end{align}

	\textit{5. Identifying the convective term:}
In the following, we, once again, set $\bv_{h_n,\tau_n}= I_{\tau_n}(\phi)\Pi_{\mathbb{V}^{h_n}}\bv\in \mathbb{V}^{h_n,\tau_n}$, for arbitrary $\phi \in C_c^\infty([0,T))$, $\bv\in C_c^\infty(\Omega)^d$, and $n\in \mathbb{N}$.  Appealing to \cite[Eq.~69.11]{EG21c}, for every $n\in \mathbb{N}$, it holds that 
  \begin{align}\label{eq:ident}
      \int_Q \bu_{h_n,\tau_n}\otimes \bu_{h_n,\tau_n} \fp \mathcal{G}_{h_n}^{2k_{\bu}}(\bv_{h_n,\tau_n}) \,\d x\, \GRdt
=
\int_Q I_{\tau_n}(\bu_{h_n,\tau_n}\otimes \bu_{h_n,\tau_n}) \fp \mathcal{G}_{h_n}^{2k_{\bu}}(\bv_{h_n,\tau_n}) \,\d x  \,\d t\,,
  \end{align}
  where the interpolant $I_{\tau_n}$ is interpreted component-wise. Now let $\varepsilon>0$ be arbitrary and take a function (e.g., obtained via mollification) $\bu_\varepsilon \in C^1(\overline{I};L^2(\Omega)^d)$ such that $\|\bu_\varepsilon - \bu\|_{L^2(Q)}< \varepsilon$. In particular, point values in time of $\bu_\varepsilon$ are defined and it is meaningful to apply the Lagrange interpolation operator $I_{\tau_n}$. From the strong $L^2$-convergence \eqref{eq:strong_L2_convergence}, for every $n\in \mathbb{N}$, it follows that 
  \begin{align*}
\|I_{\tau_n}(\bu_{h_n,\tau_n}\otimes \bu_{h_n,\tau_n}) - \bu\otimes \bu\|_{L^1(Q)}
&\leq
\|I_{\tau_n}(\bu_{h_n,\tau_n}\otimes \bu_{h_n,\tau_n}-\bu_\varepsilon\otimes \bu_\varepsilon)\|_{L^1(Q)}
+ \|I_{\tau_n}(\bu_\varepsilon\otimes \bu_\varepsilon) - \bu\otimes \bu\|_{L^1(Q)} \\
&\lesssim
\|\bu_{h_n,\tau_n}\otimes \bu_{h,\tau} - \bu_\varepsilon\otimes \bu_\varepsilon\|_{L^1(Q)}
+
\|I_{\tau_n}(\bu_\varepsilon\otimes \bu_\varepsilon) - \bu\otimes \bu\|_{L^1(Q)} \\
&\leq  \|\bu_{h_n,\tau_n}\otimes \bu_{h_n,\tau_n} - \bu\otimes \bu\|_{L^1(Q)}
+
\|\bu\otimes \bu - \bu_\varepsilon\otimes \bu_\varepsilon\|_{L^1(Q)} \\
&\quad +\|I_{\tau_n}(\bu_\varepsilon\otimes \bu_\varepsilon) - \bu_\varepsilon\otimes \bu_\varepsilon\|_{L^1(Q)}
+
\|\bu_\varepsilon\otimes \bu_\varepsilon - \bu\otimes \bu\|_{L^1(Q)}
\\&\leq  \|\bu_{h_n,\tau_n}\otimes \bu_{h_n,\tau_n} - \bu\otimes \bu\|_{L^1(Q)}\\
&\quad +\|I_{\tau_n}(\bu_\varepsilon\otimes \bu_\varepsilon) - \bu_\varepsilon\otimes \bu_\varepsilon\|_{L^1(Q)}+2\varepsilon,
  \end{align*}
  where we also used the stability and approximation properties of $I_{\tau_n}$. Thus, taking the limit superior with respect to $n\to \infty$, for arbitrary $\varepsilon>0$, we observe that
  \begin{align*}
      \limsup_{n\to \infty}{\|I_{\tau_n}(\bu_{h_n,\tau_n}\otimes \bu_{h_n,\tau_n}) - \bu\otimes \bu\|_{L^1(Q)}}\leq 2\varepsilon\,.
  \end{align*}
  In other words, we have that $I_{\tau_n}(\bu_{h_n,\tau_n}\otimes \bu_{h_n,\tau_n}) \to \bu\otimes \bu$ in $L^1(Q)^{d\times d}$ $(n\to \infty)$, whence, using \eqref{eq:ident}, it follows that
 \begin{align}
     \int_Q \bu_{h_n,\tau_n}\otimes \bu_{h_n,\tau_n} \fp  \mathcal{G}_h^{2k_{\bu}}(\bv_{h_n,\tau_n}) \,\d x\, \GRdt
\to \int_Q \bu\otimes \bu \fp \phi \nabla \bv  \,\d x\, \d t\quad (n\to \infty)\,.
\end{align}
  On the other hand, from Lemma \ref{lem:weak-star_DG}, it also follows that
	\begin{align*}
			\mathcal{G}_{h_n}\bu_{h_n,\tau_n}^{k_{\bu}}\rightharpoonup \nabla\bu\quad \text{ in }L^2(I;L^2(\Omega)^{d\times d})\quad (n\to \infty)\,.
	\end{align*}
 Thus, from an analogous argument, exploiting that
		\begin{align*}
		\begin{aligned}
			\bv_{h_n,\tau_n}&\overset{\ast}{\rightharpoondown} \phi \bv&&\quad \text{ in }L^\infty(I;L^\infty(\Omega)^d)&&\quad (n\to \infty)\,,\\
			\mathcal{G}_{h_n}^{2k_{\bu}}\bv_{h_n,\tau_n}^{k_{\bu}}&\rightharpoonup\phi \nabla\bv&&\quad \text{ in }L^\infty(I;L^\infty(\Omega)^{d\times d})&&\quad (n\to \infty)\,,
		\end{aligned}
	\end{align*}
	it is possible to pass to the limit in the second term of the discrete convective term and, therefore,
	\begin{align*}
			\int_I{\mathscr{C}_{h_n}[\bu_{h_n,\tau_n},\bu_{h_n,\tau_n},\bv_{h_n,\tau_n}]\, \GRdt}
   \to \frac{1}{2}\int_Q \phi \left[ (\bv \otimes \bu)\fp \nabla \bu - (\bu\otimes \bu)\fp \nabla \bv\right] \,\d x\, \d t\quad (n\to \infty)\,.
	\end{align*}
 In fact, since $\textup{div}\, \bu=0$ a.e.\ in $Q$, integration-by-parts actually yields that
\begin{align}
    			\int_I{\mathscr{C}_{h_n}[\bu_{h_n,\tau_n},\bu_{h_n,\tau_n},\bv_{h_n,\tau_n}]\, \GRdt}
   \to -\int_Q (\bu\otimes \bu)\fp \phi \nabla \bv \,\d x\, \d t\quad (n\to \infty)\,.
	\end{align}

  \textit{6. Identifying the constitutive relation:} 
  Choosing $(\bv_{h_n,\tau_n},q_{h_n,\tau_n})^\top\coloneqq(\bu_{h_n,\tau_n},p_{h_n,\tau_n})^\top\in \mathbb{V}^{h_n,\tau_n}\times \mathbb{M}^{h_n,\tau_n}$ in \eqref{eq:discrete_PDE}, for every $n\in \mathbb{N}$,  we find that
  \begin{align}\label{eq:iden_const_rel.1}
      \begin{aligned}
\sum_{j=1}^{N_{\tau_n}}\,\frac{1}{2}\int_{I_j} \frac{\d}{\d t}\norm{\smash{\bu_{h_n,\tau_n}}}^2_{\Lp{2}}\, \d t
&+
\int_\Omega (\bu_{h_n,\tau_n}(t^+_{j-1})- \bu_{h_n,\tau_n}(t_{j-1}))\cdot \bu_{h_n,\tau_n}(t^+_{j-1})\, \d x
\\
&\quad+
\int_I \mathcal{A}_{h_n}(\bu_{h_n,\tau_n}; \bu_{h_n,\tau_n}) \,\GRdt +
\int_I S^{\pi}_{h_n}(p_{h_n,\tau_n}, p_{h_n,\tau_n})\, \GRdt
\\&=
\int_I{\int_{\Omega}{ \bm{f} \cdot \bu_{h_n,\tau_n} \,\d x}\,\GRdt}\,.
      \end{aligned}
  \end{align}
    Then, using that fact that $2a(a-b)=a^2-b^2+(a-b)^2$ for all $a,b\in \mathbb{R}$ and the classical telescope summation trick, for every $n\in \mathbb{N}$, from \eqref{eq:iden_const_rel.1}, we deduce that
     \begin{align}\label{eq:iden_const_rel.2}
      \begin{aligned}
\frac{1}{2}\norm{\smash{\bu_{h_n,\tau_n}(T)}}^2_{\Lp{2}}&-\frac{1}{2}\norm{\smash{\bu_{h_n,\tau_n}(0)}}^2_{\Lp{2}}
+\sum_{j=1}^{N_{\tau_n}}{\frac{1}{2}\norm{\jump{\smash{\bu_{h_n,\tau_n}}}_{j-1}}^2_{\Lp{2}}}\\
&\quad
+
\int_I \mathcal{A}_{h_n}(\bu_{h_n,\tau_n}; \bu_{h_n,\tau_n}) \,\GRdt +
\int_I S^{\pi}_{h_n}(p_{h_n,\tau_n}, p_{h_n,\tau_n})\, \GRdt
\\&=
\int_I{\int_{\Omega}{ \bm{f} \cdot \bu_{h_n,\tau_n} \,\d x}\,\GRdt}\,.
      \end{aligned}
  \end{align}
  In particular, dropping several non-negative terms on the left-hand side of \eqref{eq:iden_const_rel.2} as well as using, again, that
  $\norm{\smash{\bu_{h_n,\tau_n}(0)}}_{\Lp{2}}\leq \norm{\smash{\bu_0}}_{\Lp{2}}$,
  for every $n\in \mathbb{N}$, from \eqref{eq:iden_const_rel.2}, we obtain
  \begin{align}\label{eq:iden_const_rel.3}
  \begin{aligned}
\int_I \mathcal{A}_{h_n}(\bu_{h_n,\tau_n}; \bu_{h_n,\tau_n}) \,\GRdt &\leq
\int_I{\int_{\Omega}{ \bm{f} \cdot \bu_{h_n,\tau_n} \,\d x}\,\GRdt}\\&\quad-\frac{1}{2}\norm{\smash{\bu_{h_n,\tau_n}(T)}}^2_{\Lp{2}}+\frac{1}{2}\norm{\smash{\bu_0}}^2_{\Lp{2}}\,.
      \end{aligned}
  \end{align}
  Note that, appealing to Lemma \ref{lem:parabolic_interpolation} (for $q=p=2$, $s=p$, and $r=\frac{p}{p-2}\frac{4}{d}$) together with \eqref{eq:velocity_convergences}, we have that
  \begin{align*}
      \bu_{h_n,\tau_n}\to \bu \quad\text{ in }L^r(I;L^p(\Omega)^d)\quad (n\to \infty)\,,
\end{align*}
    so that from $\bm{f} \in C^0(\overline{I};L^{p'}(\Omega)^d)$ and the properties of $I_{\tau_n}$, it follows that
    \begin{align}\label{eq:iden_const_rel.4}
     \begin{aligned}
        \int_I{\int_{\Omega}{ \bm{f} \cdot \bu_{h_n,\tau_n} \,\d x}\,\GRdt}&=\int_I{\int_{\Omega}{ I_{\tau_n}(\bm{f}) \cdot \bu_{h_n,\tau_n} \,\d x}\,\d t}\\
        &\to \int_Q \bm{f} \cdot \bu \,\mathrm{d}t\mathrm{d}x\quad (n\to \infty)\,.
         \end{aligned}
    \end{align}
  Thus, taking the limit superior with respect to $n\to \infty$ in \eqref{eq:iden_const_rel.3}, using \eqref{eq:iden_const_rel.4}, $\bu(0)=\bu_0$ in $L^2_{\textup{div}}(\Omega)$ (cf.\ \eqref{eq:conv_in_0}), \eqref{eq:conv_in_T}, integration-by-parts, and the weak formulation \eqref{eq:equation},
 we find that
  \begin{align*}
    \begin{aligned} 
    \limsup_{n\to \infty}{\int_I {\int_\Omega{\hat{\bm{S}}(\dgrad\bu_{h_n,\tau_n}):\dgrad\bu_{h_n,\tau_n}\,\d x} \,\GRdt}}&\leq\limsup_{n\to \infty}{\int_I \mathcal{A}_{h_n}(\bu_{h_n,\tau_n}; \bu_{h_n,\tau_n}) \,\GRdt}\\&\leq \int_Q \bm{f} \cdot \bu \,\mathrm{d}t\mathrm{d}x-\frac{1}{2}\norm{\smash{\bu(T)}}^2_{\Lp{2}}+\frac{1}{2}\norm{\smash{\bu(0)}}^2_{\Lp{2}}
  	\\&= \int_Q \bm{f} \cdot \bu \,\mathrm{d}t\mathrm{d}x-\left\langle \frac{\mathrm{d}\bu}{\mathrm{d}t},\bu\right\rangle_{L^q(I;X_{\textup{div}}(\Omega))}
  	\\&= \int_Q \bm{S} : \nabla\bu \,\mathrm{d}t\mathrm{d}x\,.
   \end{aligned}
  \end{align*}
  Eventually, taking the limit superior with respect to $n\to \infty$ in
  \begin{align}
  	\begin{aligned}
  		0&\leq  \int_I{ \int_{\Omega}{(\hat{\bm{S}}(\dgrad\bu_{h_n,\tau_n})-\hat{\bm{S}}(\dgrad\bv_{h_n,\tau_n})):(\dgrad\bu_{h_n,\tau_n}-\dgrad\bv_{h_n,\tau_n})\, \d x}\,\GRdt}
  		\\&=\int_I{\int_{\Omega}{ \hat{\bm{S}}(\dgrad\bu_{h_n,\tau_n}):\dgrad\bu_{h_n,\tau_n}\, \d x}\,\GRdt}
  		\\&\quad-\int_I{\int_{\Omega}{ \hat{\bm{S}}(\dgrad\bu_{h_n,\tau_n}):\dgrad\bv_{h_n,\tau_n}\, \d x}\,\GRdt}
  		\\&\quad -\int_I{ \int_{\Omega}{\hat{\bm{S}}(\dgrad\bv_{h_n,\tau_n}):(\dgrad\bu_{h_n,\tau_n}-\dgrad\bv_{h_n,\tau_n})\, \d x}\,\GRdt}\,,
  	\end{aligned}
  \end{align}
  for every $\bv=\phi\bw\in C_0^\infty(Q)^d$, where $\phi\in C_0^\infty(I)$ and $\bw\in C_0^\infty(\Omega)^d$ are arbitrary, we find that
  \begin{align*}
  	0\leq  \int_Q{ (\bm{S}-\hat{\bm{S}}(\nabla\bv)):(\nabla\bu-\nabla\bv)\,\mathrm{d}t\mathrm{d}x}\,.
  \end{align*}
so that the maximal monotonicity of $\hat{\bm{S}}$ yields that $\bm{S}=\hat{\bm{S}}(\nabla\bv)$ in $L^{p'}(I;L^{p'}(\Omega)^{d\times d})$.
\end{proof}
}
\section*{Acknowledgements}
The authors would like to thank the anonymous referees, whose comments and suggestions helped increase the quality of the manuscript.

\appendix
\section{Inf-sup stability}\label{appendix:infsup}

\hspace{5mm}In this section, we will prove the inf-sup inequality \eqref{eq:infsup}, which although not relevant in the results from this paper, is of great importance, e.g., for proving existence and stability of the discrete pressure. The proof follows the argument from \cite[Lm.\ 4.1]{dPE.2010}, where the case $p=2$ is covered.

Let $q_h\in \Mh$ be arbitrary. From the surjectivity of the divergence operator $\rmdiv\colon W^{1,p}_0(\Omega)^d \to L^p_0(\Omega)$ (see, e.g., \cite[Eq.\ III.3.2]{Gal.2011}), we know that there exists $\bv_{q_h}\in W^{1,p}_0(\Omega)^d$ such that
\begin{subequations}
\begin{align}
	\rmdiv \bv_{q_h}& = |q_h|^{p'-2}q_h - \frac{1}{|\Omega|}\int_\Omega |q_h|^{p'-2}q_h\,\d x\,, \label{eq:inverse_divergence}\\
\|\bv_{q_h}\|_{W^{1,p}(\Omega)} 
&\lesssim \norm{q_h}^{p'-1}_{L^{p'}(\Omega)}\,. \label{eq:inverse_divergence_stability}
\end{align}
\end{subequations}
Multiplying \eqref{eq:inverse_divergence} by $q_h\in \Mh $ and integrating-by-parts, we find that
\begingroup
\allowdisplaybreaks
\begin{align*}
\|q_h\|^{p'}_{L^{p'}(\Omega)}
&=
\int_\Omega q_h \rmdiv \bv_{q_h}\,\d x
\\&= 
-\int_\Omega \nabla_h q_h \cdot \bv_{q_h}\,\d x
+ \int_{\intF} \jump{q_h \bn}\cdot \bv_{q_h}\,\d s \\
&= -\int_\Omega \nabla_h q_h \cdot \Pi_{\Vh}\bv_{q_h}\,\d x
+ \int_{\intF} \jump{q_h \bn}\cdot \bv_{q_h} \,\d s \\
&= 
\int_\Omega q_h \rmdiv_h(\Pi_{\Vh}\bv_{q_h})\,\d x
+ \sum_{K\in \Th} \int_{\partial K} q_h \bn_K\cdot \Pi_{\Vh}\bv_{q_h}\,\d s
+ \int_{\intF} \jump{q_h \bn}\cdot \bv_{q_h} \,\d s \\
&=
\int_\Omega q_h \ddiv(\Pi_{\Vh}\bv_{q_h}) \,\d x
+ \int_{\intF} \jump{q_h \bn} \cdot \avg{v_{q_h} - \Pi_{\Vh}\bv_{q_h}}\,\d s\\ 
&= \mathfrak{I}_1 + \mathfrak{I}_2\,,
\end{align*}
\endgroup
where we introduced the $L^2$-orthogonal projection $\Pi_{\Vh}\bv_{q_h}$ of $\bv_{q_h}$ onto $\Vh$ (recall that $k_{\pi}\leq k_{\bu}$). Thus,
\begin{align*}
	|\mathfrak{I}_1| 
	&=
	\frac{\left|\int_\Omega q_h \ddiv(\Pi_{\Vh}\bv_{q_h})\,\d x \right|}{\norm{\Pi_{\Vh}\bv_{q_h}}_{h,p}} \norm{\Pi_{\Vh}\bv_{q_h}}_{h,p}
	\\&\lesssim
	\bigg( 
\sup_{\bw_h\in \Vh} \frac{\int_\Omega q_h \ddiv(\bw_h)\,\d x}{\|\bw_h\|_{h,p}}
	\bigg)
	\|\bv_{q_h}\|_{W^{1,p}(\Omega)} \\
	&\lesssim
	\bigg( 
\sup_{\bw_h\in \Vh} \frac{\int_\Omega q_h \ddiv(\bw_h)\,\d x}{\|\bw_h\|_{h,p}}
	\bigg)
	\|q_h\|^{p'-1}_{L^{p'}(\Omega)},
\end{align*}
where we used the stability of the $L^2$-projector $\|\Pi_{\Vh}\bv_{q_h}\|\lesssim \|\bv_{q_h}\|_{W^{1,p}(\Omega)}$ and \eqref{eq:inverse_divergence_stability}.~To~deal with $\mathfrak{I}_2$, we first note that a local inverse inequality and the approximation properties of $\Pi_{\Vh}$~(recalling~that~$h_K \lesssim h_F$) imply that
\begin{align}
h_F^{-\frac{1}{p}}
\|\bv_{q_h} - \Pi_{\Vh}\bv_{q_h}\|_{L^p(F)} 
\lesssim 
\|\bv_{q_h}\|_{W^{1,p}(K)}\,.
\end{align}
where $F\in \intF \cap \partial K$ for arbitrary $K\in \Th$. Hence,
\begin{align*}
|\mathfrak{I}_2| 
&\leq \sum_{F\in \intF} \|\jump{q_h \bn}\|_{L^{p'}(F)}
\|\avg{\bv_{q_h} - \Pi_{\Vh}\bv_{q_h}}\|_{L^p(F)} \\
&\leq 
\bigg(\sum_{F\in \intF} h_F^{\frac{p'}{p}} \|\jump{q_h \bn}\|^{p'}_{L^{p'}(F)} \bigg)^{\smash{1/p'}}
\bigg(\sum_{F\in \intF} h_F^{-1} \|\avg{\bv_{q_h} - \Pi_{\Vh}\bv_{q_h}} \|^{p}_{L^{p}(F)} \bigg)^{\smash{1/p}} \\
&\lesssim 
S^{\pi}_h(q_h; q_h)^{\frac{1}{p'}} 
\bigg(\sum_{K\in \Th} \|\bv_{q_h}\|^p_{W^{1,p}(\omega)} \bigg)^{\smash{1/p}}
\\
&\lesssim 
S^{\pi}_h(q_h; q_h)^{\frac{1}{p'}} 
\|q_h\|^{p'-1}_{L^{p'}(\Omega)},
\end{align*}
where we used the fact that the number of elements that contain a given facet on their boundary is uniformly bounded from above. This concludes the proof of \eqref{eq:infsup}.

\section{Stability of the exponential interpolant}\label{appendix:stability}
\hspace{5mm} We will now proceed to prove Lemma \ref{lem:exp_stability}. Consider first the stability estimate \eqref{eq:exp_stability_GR_p_star}. Since $\norm{\cdot}_\star$ arises from an inner product, and since the quadrature is exact up to degree $2k_t$, the result is immediate for $s=2$:
\begin{align}\label{eq:exp_stability_GR_2_star}
\int_{I_j} \|\overline{\bv}_{h,\tau}(t)\|^2_\star\, \GRdt
= 
\int_{I_j} \|\overline{\bv}_{h,\tau}(t)\|^2_\star \,\d t
\lesssim
\int_{I_j} \|\bv_{h,\tau}(t)\|^2_\star \,\d t
=
\int_{I_j} \|\bv_{h,\tau}(t)\|^2_\star \,\GRdt\,.
\end{align}
For general $s\in (1,\infty)$, we make use of inverse-type inequalities to go back to the $s=2$ case and, then, use \eqref{eq:exp_stability_GR_2_star}. Namely, we claim that for $r,s\in (1,\infty)$ we have for a function $g\in C(I_j)$ 
\begin{align}\label{eq:inverse_ineq_GR}
	\bigg( \int_{I_j} |g(t)|^r \GRdt \bigg)^{\smash{1/r}}
	\lesssim
	\tau^{\frac{r-s}{rs}}
	\bigg( \int_{I_j} |g(t)|^s \GRdt \bigg)^{\smash{1/s}}\,.
\end{align}
To see this, suppose first that $s\geq r$. Then, using Hölder’s inequality, we find that
\begin{align*}
\int_{I_j} |g(t)|^r \GRdt \leq 
\bigg(\sum_{l=1}^{k_t+1} \omega_l^j |g(\xi^j_l)|^s \bigg)^{\smash{r/s}}
\bigg(\sum_{l=1}^{k_t+1} \omega_l^j \bigg)^{\smash{(s-r)/s}}\,.
\end{align*}
Recalling that $\sum_{l=1}^{k_t+1}\omega_l^j = |I_j| \leq \tau$ yields the claim. Suppose now that $s\leq r$; assume also for the moment that $\int_{I_j}|g(t)|^s\, \GRdt \!=\! 1$; this implies for all $l\!\in\!\{1,\ldots, k_t+1\}$~that~${\omega_l^j |g(\xi_l^j)|^s \!\leq \!1}$.~Then,~since~${\frac{r}{s}\!\geq\! 1}$, one has that $\smash{(\omega_l^j)^{r/s} |g(\xi^j_l)|^r \leq \omega_l^j |g(\xi_l^j)|^s}$. Hence, we have that
\begin{align*}
	\bigg(\sum_{l=1}^{k_t} (\omega_l^j)^{\frac{r-s}{s}} \omega_l^j |g(\xi_l^j)|^r \bigg)^{\smash{1/r}} 
	\leq 1\,,
\end{align*}
and so
\begin{align*}
	\bigg(\int_{I_j} |g(t)|^r \GRdt \bigg)^{\smash{1/r}} 
	\leq 
	\bigg(\min_{l\in \{1,\ldots, k_t+1\}} \omega_l^j \bigg)^{\smash{(s-r)/rs}}
	=
	\bigg( |I_j| \min_{l\in \{1,\ldots, k_t+1\}} \omega_l \bigg)^{\smash{(s-r)/rs}}\,,
\end{align*}
where we expressed the weights in terms of those on the reference interval $\hat{I}$ (which are known data); the claim \eqref{eq:inverse_ineq_GR} then follows from homogeneity and the quasi-uniformity \eqref{eq:time_quasiuniform} of the time discretisation.

We now turn to the proof of \eqref{eq:exp_stability_p}. Denote by $\mathcal{T}_F$, the patch of elements sharing~a~facet~${F\in  \Fh}$. The first important observation, consequence of the equivalence of norms on finite-dimensional spaces and a scaling argument, is the following:
\begin{subequations}\label{eq:norm_equivalence}
\begin{align}
	\bigg(\int_{I_j} \|\bv(t)\|^s_{h,s}\, \d t \bigg)^{\smash{1/s}}
	&\hspace{-0.1em}\lesssim\hspace{-0.1em}
	\tau^{\frac{1}{s}}
	\max_{t\in I_j} \left[\|\BD_h(\bv)(t)\|_{\Lp{s}} \hspace{-0.1em}+\hspace{-0.1em} |\bv(t)|_{\Gamma_h,s} \right]\hspace{-0.1em},\label{eq:norm_equivalence1}\\ 
	\max_{t\in I_j} \left[\|\BD\bv(t)\|_{\Lp[K]{s}} \hspace{-0.1em}+ \hspace{-0.1em}\|\bv(t)\|_{\Lp[K]{s}} \right]
	&\hspace{-0.1em}\lesssim\hspace{-0.1em}
	\tau^{-\frac{1}{s}}
	\bigg(\int_{I_j} \left[\|\BD\bv(t)\|^s_{\Lp[K]{s}} \hspace{-0.1em}+ \hspace{-0.1em}\|\bv(t)\|^s_{\Lp[K]{s}} \right] \,\d t \bigg)^{\smash{1/s}} \hspace{-0.1em},\label{eq:norm_equivalence2}\\
	\max_{t\in I_j} \left[\|h_F^{\frac{-1}{s'}}\jump{\bv(t)\otimes \bn}\|_{L^s(F)} \hspace{-0.1em}+ \hspace{-0.1em}\|\bv(t)\|_{\Lp[\mathcal{T}_F]{s}} \right]
	&\hspace{-0.1em}\lesssim\hspace{-0.1em}
	\tau^{-\frac{1}{s}}
	\bigg(\int_{I_j} \left[h_F^{1-s}\|\jump{\bv(t)\otimes \bn}\|_{L^s(F)}^s \hspace{-0.1em}+\hspace{-0.1em} \|\bv(t)\|^s_{\Lp[\mathcal{T}_F]{s}} \right] \,\d t \bigg)^{\smash{1/s}}\hspace{-0.1em},\label{eq:norm_equivalence3} 
\end{align}
\end{subequations}
which holds, respectively, for $\bv$ belonging to the spaces~$\P_{k_t}(I_j;\Vh)$,~$\P_{k_t}(I_j;\P_{k_{\bu}}(K))$,~and~$\P_{k_t}(I_j;\P_{k_{\bu}}(\mathcal{T}_F))$, since each line defines norms on the respective spaces. Thus, for an arbitrary $\bv\in \P_{k_t}(I_j; \Vh)$, we obtain
\begingroup
\allowdisplaybreaks
\begin{align*}
	\int_{I_j} \|\overline{\bv}&(t)\|^s_{h,s} \,\d t
 \overset{\eqref{eq:norm_equivalence1}}{\lesssim}
\tau \max_{t\in I_j} \left[\|\BD_h(\overline{\bv})(t)\|_{\Lp{s}} + |\overline{\bv}(t)|_{\Gamma_h,s} \right]^s
\lesssim
\tau \max_{t\in I_j} \left[\|\BD_h(\overline{\bv})(t)\|^s_{\Lp{s}} + |\overline{\bv}(t)|^s_{\Gamma_h,s} \right]\\
&\lesssim
\tau \max_{t\in I_j} \left[\sum_{K\in \Th} h_K^{ds(\frac{1}{s} - \frac{1}{2})}\|\BD\overline{\bv}(t)\|^s_{\Lp[K]{2}} +
 \sum_{F\in \Fh} \frac{h_F^{(d-1)s(\frac{1}{s} - \frac{1}{2})}}{h_F^{s-1}}\|\jump{\overline{\bv}(t)\otimes \bn}\|^s_{L^2(F)} \right]\\
&\leq
\tau  \sum_{K\in \Th} h_K^{ds(\frac{1}{s} - \frac{1}{2})} \max_{t\in I_j}\|\BD\overline{\bv}(t)\|^s_{\Lp[K]{2}} +
\tau \sum_{F\in \Fh} \frac{h_F^{(d-1)s(\frac{1}{s} - \frac{1}{2})}}{h_F^{s-1}} \max_{t\in I_j}\|\jump{\overline{\bv}(t)\otimes \bn}\|^s_{L^2(F)} \\
&\overset{\mathrlap{\hspace{-1.1em}\eqref{eq:exp_stability_p_star}}}{\lesssim}
\tau  \sum_{K\in \Th} h_K^{ds(\frac{1}{s} - \frac{1}{2})} \max_{t\in I_j}\|\BD\bv(t)\|^s_{\Lp[K]{2}} +
\tau \sum_{F\in \Fh} \frac{h_F^{(d-1)s(\frac{1}{s} - \frac{1}{2})}}{h_F^{s-1}} \max_{t\in I_j}\|\jump{\bv(t)\otimes \bn}\|^s_{L^2(F)} \\
&\lesssim
\tau  \sum_{K\in \Th} \max_{t\in I_j}\|\BD\bv(t)\|^s_{\Lp[K]{s}} +
\tau \sum_{F\in \Fh} h_F^{1-s} \max_{t\in I_j}\|\bv(t)\|^s_{L^s(F)} \\
&\leq
\tau  \sum_{K\in \Th} \left(\max_{t\in I_j} [\|\BD\bv\|_{\Lp[K]{s}}  + \norm{\bv}_{L^s(K)}] \right)^s+
\tau \sum_{F\in \Fh} \left( \max_{t\in I_j}[  \| h_F^{\frac{-1}{s'}}\bv\|_{L^s(F)} + \norm{\bv}_{L^s(\mathcal{T}_F)}] \right)^s \\
&\overset{\mathrlap{\hspace{-2.2em}\eqref{eq:norm_equivalence2}\eqref{eq:norm_equivalence3}}}{\lesssim}
\tau  \sum_{K\in \Th} \bigg(\tau^{\frac{-1}{s}}\bigg[\int_{I_j} \norm{\BD\bv(t)}^s_{L^s(K)} + \norm{\bv(t)}^s_{L^s(K)} \,\d t \bigg]^{\frac{1}{s}} \bigg)^s \\
&\hphantom{aaifjas} +
\tau \sum_{F\in \Fh} \bigg( \tau^{\frac{-1}{s}} \bigg[ \int_{I_j} \|h_F^{\frac{-1}{s'}} \jump{\bv(t)\otimes \bn}\|^s_{L^s(F)} + \norm{\bv(t)}^s_{L^s(\mathcal{T}_F)} \bigg]^{\frac{1}{s}} \bigg)^s \\
&\overset{\eqref{eq:korn}}{\lesssim} \int_{I_j} \|\bv(t)\|^s_{h,s} \,\d t,
\end{align*}
\endgroup
where in the final line we also used the fact that the number of elements sharing a facet is uniformly bounded from above. This yields \eqref{eq:exp_stability_p}.

The proof of \eqref{eq:exp_stability_GR_p} follows the same reasoning as above, but where the maximum is taken over the quadrature points (i.e.,  $\max_{t\in I_j} \mapsto \max_{l\in\{1,\ldots, k_t+1\}}$). For this, we require the analogous inequalities to \eqref{eq:norm_equivalence} but integrating with respect to the discrete measure $\GRdt$; the analogous inequality to \eqref{eq:norm_equivalence1} is straightforward:
\begin{align*}
	\bigg( \int_{I_j} \|\bv(t)\|^s_{h,s} \GRdt \bigg)^{\smash{1/s}}
	&\leq 
	\bigg(\sum_{l=1}^{k_t+1}\omega_l^j \bigg)^{\smash{1/s}}
	\max_{l\in\{1,\ldots, k_t+1\}} \left[\norm{\smash{\smash{\BD_h\smash{\bv(\xi_l^j)}}}}^s_{L^s(\Omega)} +  |\smash{\bv(\xi_l^j)}|^s_{\Gamma_h,s} \right]^{\smash{1/s}} \\
	&\leq 
	\tau
	\max_{l\in\{1,\ldots, k_t+1\}} \left[\norm{\smash{\BD_h\smash{\bv(\xi_l^j)}}}_{L^s(\Omega)} +  |\smash{\bv(\xi_l^j)}|_{\Gamma_h,s} \right]\,. 
\end{align*}
To prove the analogous inequality to \eqref{eq:norm_equivalence2}, let $\tilde{l}\in \{1,\ldots, k_t+1\}$ be such that
\begin{align*}
\left[\norm{\smash{\BD\smash{\bv(\xi_{\tilde{l}}^j)}}}_{L^s(K)} + \norm{\smash{\bv(\xi_{\tilde{l}}^j)}}_{L^s(K)}\right]^s
=
\max_{l\in\{1,\ldots,k_t+1\}} \left[\norm{\BD\smash{\bv(\xi_l^j)}}_{L^s(K)} + \norm{\smash{\bv(\xi_l^j)}}_{L^s(K)}\right]^s\,.
\end{align*}
Then we have that
\begin{align*}
&	\left(\min_{l\in\{1,\ldots, k_t+1\}} \omega_{l}^j \right)\max_{l\in\{1,\ldots,k_t+1\}} \left[\norm{\BD\smash{\bv(\xi_l^j)}}_{L^s(K)} + \norm{\smash{\bv(\xi_l^j)}}_{L^s(K)}\right]^s \\
&\qquad \leq \omega_{\tilde{l}}^j  
\left[\norm{\smash{\BD\smash{\bv(\xi_{\tilde{l}}^j)}}}_{L^s(K)} + \norm{\smash{\bv(\xi_{\tilde{l}}^j)}}_{L^s(K)}\right]^s
 \lesssim \omega_{\tilde{l}}^j  
\left[\norm{\smash{\BD\smash{\bv(\xi_{\tilde{l}}^j)}}}^s_{L^s(K)} + \norm{\smash{\bv(\xi_{\tilde{l}}^j)}}^s_{L^s(K)}\right] \\
&\qquad \leq 
\sum_{l=1}^{k_t+1} \omega_l^j \left[\norm{\smash{\BD\bv(\xi_{l}^j)}}^s_{L^s(K)} + \norm{\smash{\bv(\xi_{l}^j)}}^s_{L^s(K)}\right] 
= \int_{I_j} \|\BD\bv(t)\|^s_{L^s(K)} + \|\bv(t)\|^s_{L^s(K)} \,\GRdt\,.
\end{align*}
Recalling again that $\tau \lesssim \min_{l\in\{1,\ldots, k_t+1\}} \omega_l^j$, thanks to the quasi-uniformity \eqref{eq:time_quasiuniform} and to the relation of the weights to those on the reference interval, yields the estimate \eqref{eq:norm_equivalence2}. The proof of \eqref{eq:norm_equivalence3} follows a similar argument. This concludes the proof of Lemma \ref{lem:exp_stability}.

\bibliographystyle{plain}
\bibliography{./literature}

\end{document}